\documentclass[sn-mathphys-num]{sn-jnl}% Math and Physical Sciences Numbered Reference Style 
\usepackage{graphicx}%
\usepackage{multirow}%
\usepackage{amsmath,amssymb,amsfonts}%
\usepackage{amsthm}%
\usepackage{mathrsfs}%
\usepackage[title]{appendix}%
\usepackage{xcolor}%
\usepackage{textcomp}%
\usepackage{manyfoot}%
\usepackage{booktabs}%
\usepackage{algorithm}%
\usepackage{algorithmicx}%
\usepackage{algpseudocode}%
\usepackage{listings}%
\usepackage[thinlines]{easytable}
\usepackage{mathtools}
\usepackage{comment}
\usepackage{array}
\usepackage{anyfontsize}
\usepackage{multirow}
\newcommand{\R}{\mathbb{R}}
\newcommand{\C}{\mathbb{C}}
\newcommand{\id}{\text{id}}

\renewcommand{\Re}{\mathrm{Re}\,}

\newcommand{\eps}{\varepsilon}
\renewcommand{\epsilon}{\varepsilon}

\newcommand{\modebound}{0}

\theoremstyle{plain}%
\newtheorem{theorem}{Theorem}
\newtheorem{lemma}[theorem]{Lemma}%  meant for continuous numbers
\newtheorem*{claim*}{Claim}
\newtheorem{proposition}[theorem]{Proposition}%

\theoremstyle{remark}%
\newtheorem{example}{Example}%
\newtheorem{remark}[example]{Remark}%

\theoremstyle{definition}%
\newtheorem{definition}{Definition}%
\begin{document}

\title[Mode stability of blow-up for wave maps in the absence of symmetry]{Mode stability of blow-up for wave maps in the absence of symmetry}


\author*[1]{\fnm{Max} \sur{Weissenbacher}}\email{mweissen@ic.ac.uk}

\author[2]{\fnm{Herbert} \sur{Koch}}\email{koch@math.uni-bonn.de}

\author[3]{\fnm{Roland} \sur{Donninger}}\email{roland.donninger@univie.ac.at}

\affil*[1]{\orgdiv{Department of Aeronautics}, \orgname{Imperial College London}, \orgaddress{\city{London SW7 2AZ, United Kingdom}}}

\affil[2]{\orgdiv{Mathematisches Institut}, \orgname{Bonn University}, \orgaddress{\city{53115 Bonn, Germany}}}

\affil[3]{\orgdiv{Fakult\"{a}t für Mathematik}, \orgname{University of Vienna}, \orgaddress{\city{1090 Vienna, Austria}}}


\abstract{The wave maps equation in three spatial dimensions with a spherical target admits an explicit blow-up solution. Numerical studies suggest this solution captures the generic blow-up behaviour in the backward light cone of the singularity. In this work, we establish the mode stability of this blow-up solution in the backward light cone of the blow-up point without any assumptions on the symmetries of the perturbation. We classify all smooth mode solutions for growth rates $\lambda$ with $\Re \lambda \geq 0$ and demonstrate that the blow-up solution is stable up to the mode solutions arising from the symmetry group of the wave maps equation. Our proof relies on a decomposition of the linearised wave maps equation into a tractable system of symmetry-equivariant ordinary differential equations (ODEs), utilising the representation theory of the stabiliser of the blow-up solution. We then use the quasi-solution method of Costin--Donninger--Glogi\'{c} to show the absence of non-zero smooth solutions for the resulting system of ODEs.}


\maketitle

%\textcolor{orange}{\tableofcontents}

\section{Introduction}\label{sec1}
The wave maps equation is the hyperbolic analogue of the harmonic maps equation and constitutes the simplest example of a geometric wave equation. In particle physics, the wave maps equation is known as the nonlinear sigma model~\cite{gell1960axial}. Let $(M,g)$ be a $d$-dimensional Riemannian manifold and let $\R^{1+n} = \R \times \R^n$ denote Euclidean space equipped with the standard Minkowski metric. Wave maps $u: \R^{1+n} \rightarrow M$ are extremal points of the geometric Lagrangian
\begin{equation}
    L(u) = \int_{\R^{1+n}} - \| \partial_t u \|^2_g + \| \nabla_x u \|^2_g \, dt \, dx.
\end{equation}
Equivalently, if $\kappa: U \rightarrow \R^d$ denotes a coordinate chart on $M$ and $u: \R^{1+n} \rightarrow \R^d$,
\begin{equation} \label{eq:wavemaps}
\Box u_i + \sum_{j,k=1}^d \Gamma^i_{jk}(u) \: Q_0(u_j,u_k) = 0 \qquad i=1, \dots, d
\end{equation}
where $\Box = -\partial_{tt}^2 + \Delta_x$ denotes the wave operator, the null form $Q_0$ is defined by $Q_0(f,g) = -\partial_t f \, \partial_t g + \nabla_x f \cdot \nabla_x g$ and $\Gamma$ denotes the Christoffel symbols of $M$ in the chart $\kappa$.

\begin{comment}
\begin{equation}
\Box u:= (\partial_t^2- \Delta)u  = u ( |\nabla u|^2 - u_t^2). 
\end{equation} 
\end{comment}

\subsection{The blow up solution}
When $n=3$ and $M= S^3$, there exists an explicitly known self-similar blow-up solution to the wave maps equation~\eqref{eq:wavemaps}. With the coordinate chart $\kappa$ chosen as stereographic projection, the solution takes the simple form
\begin{equation}
u_0: [0,1) \times \R^3 \rightarrow \R^3, \quad u_0(t,x)=\frac{x}{1-t}.
\end{equation}
The solution $u_0$ was found in closed form by Turok--Spergel~\cite{turok}, after Shatah~\cite{shatah} proved its existence through the use of a variational argument. Numerical analysis suggests that $u_0$ describes the \emph{generic blow-up behaviour} in the backward light cone of the blow-up point~\cite{bizon}. In particular, this suggests that the solution $u_0$ is stable in the backward light cone of its blow-up point in a suitable sense. In this work, we establish that $u_0$ is mode stable modulo its symmetry group. We note that even though $u_0$ is not compactly supported, one may easily construct a blow-up solution with compactly supported initial data with an identical blow-up profile by exploiting the finite speed of propagation of the wave maps equation.

\subsection{Symmetries of the wave maps equation}
Special attention must be paid to the symmetries of the wave maps equation. The wave maps equation for $n=3$ and $M= S^3$ possesses a $17$-parameter group of symmetries, see the detailed discussion in Section~\ref{sec:symmetries}. When applied to the blow-up solution $u_0$, the symmetry group generates a family of blow-up solutions. Therefore the solution $u_0$ can only be stable modulo the symmetry group of the equation. As an example, consider the time translation symmetry $(t,x) \mapsto (t+t_0, x)$, which will either generate a solution which is regular in the backward light cone of $(t,x) = (1,0)$, or a solution which blows up at an earlier time, depending on the sign of $t_0 \in \R$. We note here that the blow-up solution is invariant under the action of $SO(3)$ given by
\begin{equation} \label{eq:SO3_intro}
    u(t,x) \mapsto R u (t , R^T x),
\end{equation}
and the scaling around the point $(t,x) = (1,0)$ given by
\begin{equation} \label{eq:scaling_intro}
    u(t,x) \mapsto u( 1+\lambda (t-1),\lambda x). 
\end{equation}
The symmetries~\eqref{eq:SO3_intro} and~\eqref{eq:scaling_intro} generate the stabiliser of $u_0$. The $13$-dimensional quotient of the full symmetry group modulo the stabiliser of $u_0$ generate mode solutions. In Theorem~\ref{maintheorem} we will show that these are the only mode solutions with growth rates $\Re \lambda \geq 0$.

\begin{comment}
which in self-similar coordinates becomes 
\[  u( \tau, y) \mapsto u( \tau-\log \lambda, y). \]
\end{comment}

\subsection{Self-similar coordinates}
We shall restrict our attention to studying the blow-up solution $u_0$ in the backward light cone of its blow-up point $(t,x) = (1,0)$, defined by
\begin{equation}
    C_{(1,0)} = \{(t,x) \in \R^{1+3} \mid 0 \leq t \leq 1, \| x \| \leq 1-t \},
\end{equation}
where $\| \cdot \|$ denotes the Euclidean norm on $\R^3$. In order to facilitate the study of $u_0$ in the backward light cone $C_{(1,0)}$, we introduce self-similar coordinates on $\R^{1+3}$ given by
\begin{equation} \label{selfsimcoordintro}
\tau = -\mathrm{log}(1-t), \quad y=\frac{x}{1-t},
\end{equation}
where $(t,x)$ denote the usual Cartesian coordinates on $\R^{1+3}$. In $(\tau, y)$-coordinates, the backward light cone $C_{(1,0)}$ is transformed into the cylinder $(0, \infty) \times B_1(0) \subset \R^{1+3}$.  In these coordinates 
\begin{equation}
    \Box = - \partial^2_{\tau \tau} - \partial_\tau - 2 \partial^2_{\tau r} - r^2 \partial^2_{rr} - 2 r \partial_r + \Delta_y. 
\end{equation}    
where $r = \left| y \right|$. For the target manifold $S^3$, we use stereographic projection from the south pole $S=(0,0,0,-1) \in S^3$.  The Christoffel symbols are 
\begin{equation}
    \Gamma^i_{jk}(z) = \frac{2}{1+ \left| z \right|^2} \left( z_i \delta_{jk} - z_j \delta_{ik} - z_k \delta_{ij} \right) .
\end{equation} 
Expressed in self-similar coordinates~\eqref{selfsimcoordintro}, the blowup solution then becomes the stationary solution
\begin{equation}
    u_0(\tau, y) = y
\end{equation}
to the wave maps equation
\begin{equation}\label{eq:wm}
 \Box u_i + \frac{2}{1+ \left| u \right|^2} \Big( u_i \sum_{j=1}^3 Q_0(u_j, u_j) - Q_0(u_i, \left| u \right|^2) \Big) = 0, \qquad i= 1,\dots,3,
\end{equation}
where (compare \eqref{eq:wavemaps})
\begin{equation}
    Q_0(f,g) = - \partial_\tau f \, \partial_\tau g + \nabla_y f \cdot \nabla_y g - r \partial_\tau f \, \partial_r g - r \partial_\tau g \, \partial_r f - (r \partial_r f) (r \partial_r g).
\end{equation}    
We note carefully that in this coordinate chart, $u_0(C_{(1,0)}) = B_1(0) \subset \R^3$, which corresponds to the upper half-sphere of $S^3$. Since we are only interested in studying the solution $u_0$ and small perturbations thereof in the backward light-cone $C_{(1,0)}$, the equation is well-posed when using stereographic projection as a chart on $S^3$. We will make repeated use of both the standard Cartesian coordinates $(t,x)$ and the self-similar coordinates $(\tau,y)$ on $\R^{1+3}$, while we use the stereographic projection from the south pole on the target manifold $S^3$ throughout the work.

\subsection{Mode stability}
Linearising the wave maps equation~\eqref{eq:wavemaps} around the solution $u_0$ yields
\begin{equation} \label{equ::linintro}
\Box \Psi + 2 \Gamma[u_0] Q_0(u_0,\Psi) + \nabla \Gamma[u_0] \cdot \Psi \: Q_0(u_0,u_0) = 0
\end{equation}
which we consider in self-similar coordinates and the stereographic projection \eqref{eq:wm}. In the present work, we consider the problem of establishing the mode stability of the blow-up solution $u_0$. The concept of mode solutions is introduced in Definition~\ref{def:modesoln}.

\begin{definition} \label{def:modesoln}
We call $\Psi: B_1(0) \rightarrow \R^3$ a mode solution to~\eqref{equ::linintro} with growth rate $\lambda \in \C$ if $\Psi \in C^{\infty}(\overline{B}_1(0))\setminus\{0\}$ and $e^{\lambda \tau} \Psi(y)$ solves equation~\eqref{equ::linintro} in self-similar coordinates.
\end{definition}

We are now ready to state our main Theorem~\ref{maintheorem}, which guarantees the mode stability of the blow-up solution $u_0$ modulo the symmetry group of the wave maps equation. Theorem~\ref{maintheorem} precisely characterizes all mode solutions with growth rates $\Re \lambda \geq 0$.

\begin{theorem} \label{maintheorem}
Assume $\Re \lambda \geq 0$ and $\Psi$ is a mode solution as in Definition~\ref{def:modesoln}. Then $\lambda \in \{0,1\}$. The space of mode solutions with growth rate $\lambda=1$ is four-dimensional and spanned by the functions $\{ \Psi_{1,0}, \Psi_{0,1}^1, \Psi_{0,1}^2, \Psi_{0,1}^3 \}$, where
\begin{equation}
\Psi_{1,0}(y)=y, \quad \Psi_{0,1}^i(y) = e_i, \; i=1,2,3
\end{equation}
where $e_i$ denote the standard Euclidean basis of $\R^3$. The space of mode solutions with growth rate $\lambda=0$ is nine-dimensional and spanned by $\{ \Phi_{0,1}^i, \Psi_{1,1}^i,\Psi_{2,1}^i \mid i=1,2,3 \}$, where
\begin{equation}
\begin{gathered}
\Phi_{0,1}^i(y)= (\| y \|^2-3) e_i, \; i=1,2,3 \\[0.2cm]
\Psi_{1,1}^1(y)=\left( \begin{matrix} 0 \\ -y_3 \\ y_2 \end{matrix} \right), \quad \Psi_{1,1}^2(y)=\left( \begin{matrix} y_3 \\ 0 \\ -y_1 \end{matrix} \right), \quad \Psi_{1,1}^3(y)=\left( \begin{matrix} -y_2 \\ y_1 \\ 0 \end{matrix} \right) \\[0.2cm]
\Psi_{2,1}^1(y)=\left( \begin{matrix} -2y_1^2+y_2^2+y_3^2 \\ -3y_1 y_2 \\ -3 y_1 y_3 \end{matrix} \right), \; \Psi_{2,1}^2(y)=\left( \begin{matrix} -3y_1 y_2 \\ y_1^2 -2y_2^2 + y_3^2 \\ -3 y_2 y_3 \end{matrix} \right), \\
\Psi_{2,1}^3(y)=\left( \begin{matrix} -3 y_1 y_3 \\ -3y_2 y_3 \\ y_1^2 + y_2^2 - 2 y_3^2 \end{matrix} \right).
\end{gathered}
\end{equation}
All basis functions are expressed in self-similar coordinates as defined in equation~\eqref{selfsimcoordintro}.
\end{theorem}

\begin{remark}
The key steps in the proof of Theorem~\ref{maintheorem} may be summarised as:
\begin{itemize}
    \item The introduction of an appropriate basis, which allows one to decouple the linearised wave maps equation~\eqref{equ::linintro} into a system of ordinary differential equations in such a way that the effect of symmetries is separated out, while still remaining tractable. The decomposition is obtained by carefully analysing the action of the symmetry group on the blow-up solution and making use of the representation theory of $\text{so}(3)$. The decomposition and the resulting system of decoupled ODEs is introduced in Section~\ref{sec:symmetries}.
    \item The use of the supersymmetric (SUSY) transformation, in order to `remove' the mode solutions generated by the action of the symmetry group. This transformation is carried out in Section~\ref{sec:removingsols}. The kernel of the transform is characterised, so that if the transformed system of ODEs possesses only the zero solution, the original system of ODEs possesses only the mode solutions generated by the symmetry group.
    \item The use of the \emph{quasi-solution method} to prove that the system of ODEs resulting from the SUSY transformation possesses only the zero solution. The quasi-solution method was originally introduced by~\citet{glogic2}. We first reduce the resulting ODEs to Heun or hypergeometric standard form in Section~\ref{sec:heunform} and then prove the absence of non-zero solutions in Sections~\ref{sec:modestability_simplecase} and~\ref{sec:proofs}.
\end{itemize}

\end{remark}

\begin{remark}
In Lemma~\ref{sec:modesolutionslemma} we make the link between the mode solutions from Theorem~\ref{maintheorem} and the action of the Lie algebra of the symmetry group on the linearised wave maps equation. In particular, we show that the space of modes for $\lambda = 1$ is generated by the spacetime translations, while the space of modes for $\lambda = 0$ is generated by rotations on $S^3$ and the Lorentz boosts. We also note that while the indexing of the basis functions in Theorem~\ref{maintheorem} may seem slightly cumbersome at first sight, this is done to facilitate comparison with the decomposition introduced in Section~\ref{sec:symmetries}.
\end{remark}

\subsection{Related work}

\subsubsection{The co-rotational case}
A commonly studied subclass of solutions is given by the co-rotational solutions. When using stereographic projection from the sphere, the co-rotational solutions are most naturally represented in the form
\begin{equation} \label{eqn:corotational}
    u(t,x)= \mathrm{tan} \left( \frac{\varphi(t,r)}{2} \right) \frac{x}{r},
\end{equation}
where $r=\|x \|$ and $\varphi(t,r)$ is a smooth function of time and radius. We note carefully that the blow-up solution $u_0$ is co-rotational, and may be expressed in the form~\eqref{eqn:corotational} if we set $\varphi(t,r) = \varphi_0(t,r) = 2 \arctan \left( \frac{r}{1-t} \right)$. For co-rotational solutions, the wave maps equation reduces to the radial semi-linear wave equation given by
\begin{equation} \label{wmcor}
\varphi_{tt} -\varphi_{rr} -\frac{2}{r} \varphi_r + \frac{\mathrm{sin}(2\varphi)}{r^2} =0 .
\end{equation}
Under the assumption of co-rotationality, the symmetry group of the wave maps equation reduces to only the time shift symmetry, which simplifies the analysis of equation~\eqref{wmcor}.

In the co-rotational case, mode stability has been shown by Costin, Donninger and Glogi\'{c} using the quasi-solution method~\cite{glogic2}, see also~\cite[Section 2.7]{glogicthesis} and \cite{Don24}. Donninger, Schörkhuber and Aichelburg~\cite{donn} demonstrated that mode stability ensures linear stability, and Donninger~\cite{donn2} proved nonlinear stability in the co-rotational case. Consequently, the co-rotational case is fully understood, and the link between mode stability and both linear and nonlinear stability in this case suggests that a similar approach may apply when the assumption of co-rotationality is removed. We shall recover the co-rotational situation as a special case in the course of our analysis, see Remark~\ref{rem:corotational_case}.

\subsubsection{Nonlinear stability}
Based on the nonlinear stability result \cite{donn2}, there was a lot of progress on self-similar blow-up for wave equations in recent years. First of all, the stability theory of \cite{donn, donn2} was generalized to related models like the scalar wave equation \cite{DonSch12, DonSch14, DonSch17}, wave maps in higher dimensions \cite{ChaDon17, DonGlo19, Glo25}, Yang-Mills equations \cite{Don14, CosDonGloHua16, Glo22, Glo24}, and Skyrmions \cite{McN20, CheMcNSch23, McN24} under suitable symmetry reductions. Furthermore, for the scalar wave equation, the nonlinear stability of self-similar blow-up is also understood without symmetry assumptions \cite{MerZaa07, DonSch16, MerZaa16, ChaDon19, GloSch21, CsoGloSch24, Ost24}. We remark that for the problem at hand, Theorem \ref{maintheorem} is the crucial stepping stone for proving nonlinear stability without symmetry assumptions along the lines of \cite{DonSch16}. This will be pursued elsewhere. In addition to these rigorous works, there is a number of very influential papers that employ numerical or mixed analytical and numerical techniques, see e.g.~\cite{BizTab01, Biz02, BizChmTab04, Biz05} and the aforementioned \cite{bizon}.

Another line of research was started in \cite{Don17} and concerns the question of \emph{optimal blow-up stability} where one studies the stability of blow-up under perturbations whose smallness is measured in the critical Sobolev norm. This encompasses the largest possible set of perturbations that is still admissible by the local well-posedness theory of the equation, see below. Results of this type exist for scalar wave equations \cite{Don17, DonRao19, Wal23} and co-rotational wave maps \cite{DonWal23, DonWal25}. 

In \cite{BieDonSch21, DonOst23, DonOst24, CheDonGloMcNSch24}, more general coordinate systems were introduced to study self-similar blow-up  and new types of stability results were established that yield information also after the blow-up. This touches upon the intriguing question of continuation beyond the blow-up which is still wide open. Another active direction, which was pioneered in \cite{Bri20}, is the study of blow-up stability under randomized perturbations.

Finally, in a slightly broader context, there is a sizeable literature on the construction and stability of \emph{non-self-similar} blow-up for wave equations, see e.g.~\cite{KriSchTat08, KriSchTat09, KriSchTat09a, RodSte10, RapRod12, HilRap12, Jen17, GhoIbrNgu18, Col18, KriMia20, BurKri22, JenKri25}. We particularly point out the work \cite{KriMiaSch24} which is more closely related to the present paper in that it proves a nonlinear stability result of non-self-similar wave maps blow-up without symmetry assumptions.

\subsubsection{Well-posedness and regularity}
The scaling symmetry $u(t,x) \mapsto u(\mu t, \mu x)$, $\mu >0$ of the wave maps equation~\eqref{eq:wavemaps} implies that $s_c= \frac{n}{2}$ is the critical regularity, in the sense that the $\dot{H}^{s_c}$ seminorm is preserved by scaling. \citet{klainermanselberg1} and~\citet{klainermanmachedon,klainerman1,klainerman2} have shown local well-posedness of the intrinsic form of the wave maps equation for initial data in $H^s \times H^{s-1}(\R^{n})$, with $s>\frac{n}{2}$ and $n \geq 2$. \citet{tataru1,tataru2} has shown well-posedness in the homogeneous critical $1$-Besov space $\dot{B}^{\frac{n}{2}}_{2,1}$ for all $n \geq 2$. The proof of well-posedness in the critical homogeneous Sobolev space $\dot{H}^{\frac{n}{2}} \times \dot{H}^{\frac{n}{2}-1}$ goes back among others to~\citet{tao1,tao2}, \citet{klainrod}, \citet{nahmod2001well}, \citet{struwe}, \citet{krieger1,krieger2} and ~\citet{tataru3}. For a comprehensive review we refer the reader to the article by~\citet{krieger}.

\subsection{Overview}
In Section~\ref{sec:symmetries} we introduce the full symmetry group of the wave maps equation and discuss how to decouple the linearised wave maps equation~\eqref{equ::linintro} into a system of ordinary differential equations (ODEs) in a symmetry-equivariant way. In addition, we show how the symmetry group generates a set of mode solutions to the linearised equation and compute these solutions explicitly. In Section~\ref{sec:removingsols}, we then discuss how to `remove' the mode solutions generated by the symmetry group in order to simplify the subsequent analysis. In Section~\ref{sec:heunform}, we show that the resulting ODEs may be transformed into a Heun-type equation in all but one cases, where the equation reduces to the simpler hypergeometric equation. The simpler hypergeometric case is treated in Section~\ref{sec:modestability_simplecase}. The remaining cases are more involved and make use of the quasi-solution method, which is elaborated in Section~\ref{sec:proofs}, concluding  the proof of Theorem~\ref{maintheorem}.

\subsection{Acknowledgments}
This research was funded in whole or in part by the Austrian Science Fund (FWF) 10.55776/P34560. H.K. was partially supported by the Deutsche Forschungsgemeinschaft (DFG, German Research Foundation) through the Hausdorff Center
for Mathematics under Germany’s Excellence Strategy - EXC-2047/1 - 390685813
and through CRC 1060 - project number 211504053. For open access purposes, the authors have applied a CC BY public copyright license to any author-accepted manuscript version arising from this submission.

\section{Symmetries and decoupling the equation}
\label{sec:symmetries} 
The group of symmetries of the wave maps equation is generated by the symmetries of the wave equation (the Poincar\'e group of dimension 10, the semi-direct product of 4 dimensional space-time translations and the Lorentz group $O(3,1)$ of dimension $6$), the isometries of the target space $S^3$ (the orthogonal group $O(4)$ of dimension $6$) and an additional scaling symmetry of the wave maps equation. The stabilizer of the blow-up solution $u_0$ is the commutative product of group $O(3)$  which acts by 
\[  u \mapsto     Ru(t, R^T x) \]
for $ R \in O(3)$ and the scaling group $(\R^+,*)$. The stabilizer  commutes with  the linearisation at $u_0$. 
We diagonalize with respect to the stabilizer: 
Diagonalization with respect to scaling leads to a three dimensional PDE in space, and after  diagonalization of the action of $SO(3)$
we obtain ordinary differential equations with respect to the radial variable for the modes.

When applied to the blow-up solution $u_0$, the symmetry group generates mode solutions of the linearised equation~\eqref{equ::linintro}. The dimension of the modes ($13$, $4$ unstable and $9$ neutral) generated by the symmetry is the dimension of the symmetry group 17 minus the dimension of the stabilizer 4. 
It is therefore crucial that we make an effort to decouple the linearised equation in such a way that allows us to separate off the mode solutions generated by symmetries. The key insight in this work is that by embedding into the larger function space $L^2(\R^{1+3}; \C^3)$, one may leverage the representation theory of the Lie algebra $so(3)$ to compute a decomposition of the solution space which is equivariant under the rotational symmetries of the linearised wave maps equation. This decomposition can furthermore be computed explicitly by using Clebsch--Gordan coefficients, which are traditionally used in angular momentum coupling in quantum mechanics. This equivariant decomposition may then be used to decouple the linearised equation~\eqref{equ::linintro} into a system of infinitely many ordinary differential equations, whose analysis will be the subject of the remaining chapters of this work.

In Section~\ref{sec:symmetries_and_modes} we discuss in some detail the symmetry group and the mode solutions generated by symmetries. In Section~\ref{sec:clebschgordan} we introduce the decomposition of the function space $L^2(\R^{1+3}; \C^3)$ which is equivariant under the symmetry group of the linearised wave maps equation. In Section~\ref{sec:decoupling} we then decouple the linearised wave maps equation using the decomposition obtained in Section~\ref{sec:clebschgordan}.

\subsection{Mode solutions generated by symmetries} \label{sec:symmetries_and_modes}
The symmetries of the wave maps equation generate mode solutions of the linearised wave maps equation. The reader may readily verify that if $u: \R^{1+3} \rightarrow \R^3$ denotes a solution to the wave maps equation, the following functions are again solutions:
\begin{align}
\text{Time translations} \qquad & \;  u(t+\alpha,x) \label{eq:translation} \\
\text{Spatial translations} \qquad & \; u(t,x+ \alpha e_i), \quad i=1,2, 3 \\
\text{Scaling} \qquad & \; u(1+\alpha (t-1), \alpha x) \quad (\alpha >0) \\
\text{Rotations} \qquad & \; R_i u(t,R_i^T(\alpha) x), \quad i=1,2, 3 \label{eq:rotdomain} \\
\text{Lorentz boosts} \qquad & \; u( e_0 +\Lambda_i(\alpha) (X-e_0)), \quad i=1,2,3 \label{eq:boost} \\
\text{Rotations on the sphere} \qquad & \; \kappa \mathbf{R}_j(\alpha) \kappa^{-1} u(t,x), \quad j=1 \dots 6.  \label{equ::rotationsonsphere}
\end{align}
Here $\alpha \in \R$ is a parameter, $e_{i}$ denote the standard unit vectors in $\R^3$ for $i=1,2,3$, $R_i(\alpha)=\exp(\alpha F_i)$ where $F_i$ generate the Lie algebra $\text{so}(3)$, $\mathbf{R}_j(\alpha) = \exp{(\alpha \mathbf{F}_j)}$, where $\mathbf{F}_j$ generate the Lie algebra $\text{so}(4)$ and $\Lambda_i(\alpha)$ are the Lorentz boosts along the $e_i$-axes with rapidity $\alpha$, which we define with respect to $e_0 = (1,0,0,0) \in \R^{1+3}$. Here we denote $X=(t,x) \in \R^{1+3}$ and $\kappa$ denotes the stereographic projection from the south pole of $S^3$. For completeness, the explicit forms of the generators of $\text{so}(3)$ and $\text{so}(4)$ as well as the Lorentz boosts are given in Appendix~\ref{appendix::matrices}. By taking a derivative in the parameter $\alpha$, one may obtain a solution to the linearised equation. We formalise this in Lemma~\ref{sec:modesolutionslemma}.

\begin{lemma} \label{sec:modesolutionslemma}
The space-time translations~\eqref{eq:translation} generate four smooth linearly independent mode solutions for $\lambda=1$ given in self-similar coordinates by
\begin{equation}
\Psi_{1,0}(y)=y, \quad \Psi_{0,1}^i(y) = e_i, \; i=1,2,3.
\end{equation}
The Lorentz boost~\eqref{eq:boost} and the rotations on the sphere~\eqref{equ::rotationsonsphere} 
generate  nine smooth linearly independent mode solutions for $\lambda=0$ given by
\begin{equation}
\begin{gathered}
\Phi_{0,1}^i(y)= (r^2-3) e_i , \quad \Psi_{1,1}^i(y)= F_i \, y \qquad i=1,2,3 \\[0.2cm]
\Psi_{2,1}^1(y)=\left( \begin{matrix} -2y_1^2+y_2^2+y_3^2 \\ -3y_1 y_2 \\ -3 y_1 y_3 \end{matrix} \right), \; \Psi_{2,1}^2(y)=\left( \begin{matrix} -3y_1 y_2 \\ y_1^2 -2y_2^2 + y_3^2 \\ -3 y_2 y_3 \end{matrix} \right), \\
\Psi_{2,1}^3(y)=\left( \begin{matrix} -3 y_1 y_3 \\ -3y_2 y_3 \\ y_1^2 + y_2^2 - 2 y_3^2 \end{matrix} \right).
\end{gathered}
\end{equation}
\end{lemma}
\begin{proof}
We compute the effect of each symmetry on the blow-up solution and compute the corresponding linear solution by taking the derivative with respect to the symmetry parameter. \\

\noindent
\textbf{Scaling.}
%Consider the one-parameter family given by scaling $u_{\alpha}(t,x)=u_0(\alpha t, \alpha x) = \frac{\alpha x}{1- \alpha t} = u_0(t+1-\frac{1}{\alpha},x)$. Observe that due to the self-similarity of our solution, scaling becomes a time shift. We take a derivative at $\alpha=1$ and obtain
The scaling group is in the stabilizer group of the self-similar solution and does not create mode solutions. \\
%\begin{equation}
%\partial_{\alpha}|_{\alpha=1} \, u_{\alpha} = \frac{x}{1-t} + \frac{t x}{(1-t)^2} = \frac{x}{(1-t)^2} = e^{\tau} y.
%\end{equation}
%which is a mode solution for mode $\lambda=1$. We set $\Psi_{1,0}(y):=y$. \\

\noindent
\textbf{Spacetime translations.} Consider the one-parameter family given by time translation, $u_{\alpha}(t,x)=u_0(t+\alpha,x)= \frac{x}{1-(t+\alpha)}$. We find
\begin{equation}
\partial_{\alpha}|_{\alpha=0} \, u_{\alpha} = \frac{x}{(1-t)^2}=e^{\tau} y.
\end{equation}
%This generates the same solution as the scaling symmetry, which is natural in light of the fact that scaling is just a time shift (using a different parametrization) for our self-similar solution. 
For the one-parameter families generated by the spatial translations $u^i_{\alpha}(t,x)=u_0(t,x+\alpha e_i)$ we find the three mode solutions
\begin{equation}
\partial_{\alpha}|_{\alpha=0} \, u^i_{\alpha} = \frac{e_i}{1-t} = e^{\tau} e_i, \quad i=1,2,3.
\end{equation}
We define $\Psi_{0,1}^i(y) := e_i$ for $i=1,2,3$ and $\Psi_{1,0}(y) = y$. \\

\noindent
\textbf{Rotations in $\R^3$.} By our definition rotations in $\R^3$ are in the stabilizer group of the selfsimilar solution and they do not create mode solutions. \\  

%Consider the three one-parameter families generated by the rotations in $\R^3$. They are given by
%\begin{equation}
%    u^i_{\alpha}(t,x)=u_0(t,R_i(\alpha) \, x)= \frac{R_i(\alpha) \, x}{1-t}= \frac{\exp(\alpha F_i) \, x}{1-t},
%\end{equation}
%so that we may readily compute
%\begin{equation}
%\partial_{\alpha}|_{\alpha=0} \, u^i_{\alpha} = \frac{F_i \, x}{1-t} = F_i \, y =: \Psi_{1,1}^i(y), \quad i=1,2,3.
%\end{equation} \\

\noindent
\textbf{Lorentz boosts.} Consider the three one-parameter families generated by the Lorentz boosts $\Lambda_i$ and given by $u^i_{\alpha}(t,x)=u_0((1,0,0,0) + \Lambda_i(\alpha) (t-1,x))$, $i=1,2,3$. After some computation one finds
\begin{equation}
\partial_{\alpha}|_{\alpha=0} \, u^i_{\alpha} = e_i - y_i y =: \tilde{\Psi}_1^i(y), \quad i=1,2,3.
\end{equation}

\noindent
\textbf{Rotations on the sphere $S^3$.}
Consider the six one-parameter families generated by the rotations on the sphere as follows: $u^j_{\alpha}(t,x)=\kappa \mathbf{R}_j(\alpha) \kappa^{-1} u_0(t,x)$ for $j=1,\dots 6$. We compute the partial derivative by $\alpha$:
\begin{equation}
\partial_{\alpha}|_{\alpha=0} \, u^j_{\alpha} = D \kappa|_{\kappa^{-1}(u_0)} \; \partial_{\alpha}|_{\alpha=0} \, \mathbf{R}_j(\alpha) \; \kappa^{-1} u_0
\end{equation}
where $\partial_{\alpha}|_{\alpha=0} \, \mathbf{R}_j(\alpha)= \partial_{\alpha}|_{\alpha=0} \, \exp(\alpha \mathbf{F}_j)=\mathbf{F}_j$ and a computation gives
\begin{equation}
D \kappa|_{\kappa^{-1}(u)} = (1+|u|^2) \left(
\begin{array}{cccc}
 \frac{1}{2} & 0 & 0 & -\frac{1}{2} u_1 \\
 0 & \frac{1}{2} & 0 & -\frac{1}{2} u_2 \\
 0 & 0 & \frac{1}{2} & -\frac{1}{2} u_3 \\
\end{array}
\right)
\end{equation}
Putting everything together, we find that for $j=1,2,3$ and $\mathbf{F}_j \in so(3)$ with a slight abuse of notation we obtain 
\begin{equation}
\partial_{\alpha}|_{\alpha=0} \, u^j_{\alpha} = \frac{\mathbf{F}_j \, x}{1-t} = \mathbf{F}_j \, y =: \Psi_{1,1}^j(y), \quad j=1,2,3,
\end{equation}
and for $j=4,5,6$ we obtain %the three new solutions
\begin{gather}
\tilde{\Psi}_2^1(y):=\partial_{\alpha}|_{\alpha=0} \, u^4_{\alpha} = \left( \begin{matrix} -y_1^2+y_2^2+y_3^2-1 \\ -2y_1 y_2 \\ -2 y_1 y_3 \end{matrix} \right), \\
\tilde{\Psi}_2^2(y):= \partial_{\alpha}|_{\alpha=0} \, u^5_{\alpha} = \left( \begin{matrix} -2y_1 y_2 \\ y_1^2-y_2^2+y_3^2-1 \\ -2y_2 y_3 \end{matrix} \right), \\
\tilde{\Psi}_2^3(y):= \partial_{\alpha}|_{\alpha=0} \, u^6_{\alpha} = \left( \begin{matrix} -2y_1 y_3 \\ -2y_2 y_3 \\ y_1^2+y_2^2-y_3^2-1 \end{matrix} \right).
\end{gather}
For later convenience, we make a change of basis for the eigenfunctions for the growth rate $\lambda=0$ generated by the Lorentz boosts and the rotations on the sphere. This will be convenient in Section~\ref{sec:clebschgordan}. For $i=1,2,3$, we define
\begin{gather}
\Psi^i_{2,1} :=  \tilde{\Psi}_1^i+\tilde{\Psi}_2^i, \\
\Phi_{0,1}^i :=  -2 \tilde{\Psi}_1^i+\tilde{\Psi}_2^i.
\end{gather}
A simple computation reveals that this yields the mode solutions in the formulation of Lemma~\ref{sec:modesolutionslemma}. It is readily verified that the solutions are linearly independent.
\end{proof}

\subsection{Clebsch--Gordan decomposition} \label{sec:clebschgordan}

In order to motivate the decomposition we use, consider the following transformation: Let $R \in \mathrm{SO}(3)$ be a rotation and let $u: \R^{1+3} \rightarrow \R^3$ be a smooth solution to the wave maps equation, expressed in Cartesian coordinates on $\R^{1+3}$ and using the stereographic projection from the south pole on $S^3$ as above. We then define
\begin{equation} \label{eqn:pi_transform_informal}
    u_R: \R^{1+3} \rightarrow \R^3, (t,x) \mapsto R u (t, R^T x).
\end{equation}
It may then be readily verified that $u_R$ is again a solution to the wave maps equation. Furthermore, it is immediately apparent that the blow-up solution $u_0$ is invariant under the transformation~\eqref{eqn:pi_transform_informal}. When viewed as a representation of the Lie group $\mathrm{SO}(3)$, this implies that the linearised wave maps equation around $u_0$ is invariant under the induced representation $\pi$ of the Lie algebra $\mathrm{so}(3)$. This invariance manifests itself in the fact that the angular derivatives in the linearised wave maps equation become diagonal when restricted to irreducible subspaces of the representation $\pi$. More precisely, we will show in Section~\ref{sec:decoupling} that the terms
containing angular derivatives  in the linearised wave map equation together act by multiplication when restricted in the angular variable to irreducible subspaces of the representation of so(3).

\begin{lemma} \label{lem:decomp}
Consider the densely defined operator $\mathfrak{M}: C^\infty(S^2; \C^3) \rightarrow L^2(S^2; \C^3)$
\begin{equation}
\mathfrak{M}= \sum_{i,j,k,a,b=1}^3 \epsilon_{ijk} \epsilon_{iab} x_j \partial_k E_{a,b},
\end{equation}
where $\epsilon_{ijk}$ denotes the Levi--Civita symbol, $E_{a,b}$ denotes the standard basis of the space of $3 \times 3$ matrices and the product $x_j \partial_k E_{a,b}$ is understood to be the matrix operator with entries $(x_j \partial_k E_{a,b})_{c,d} = \delta_{ac} \delta_{bd} x_j \partial_k$. Further define $\mathfrak{C}: C^\infty(S^2; \C^3) \rightarrow L^2(S^2; \C^3)$ by
\begin{equation}
    \mathfrak{C} = - \Delta_{S^2} + 2 +2 \mathfrak{M},
\end{equation}
where $\Delta_{S^2}$ denotes the (negative-definite) Laplace--Beltrami operator on $S^2$ acting component-wise. Then there exists a direct sum decomposition
\begin{equation} \label{eqn:L2decomp}
L^2(S^2;\C^3) =  W_{0,1} \oplus \bigoplus_{l\geq 1,|m-l| \leq 1} W_{l,m},
\end{equation}
such that each of the subspaces $W_{l,m}$ are mutually orthogonal in $L^2(S^2;\C^3)$ and every $\psi \in W_{l,m} \cap C^\infty(S^2; \C^3)$ satisfies
\begin{align}
    \Delta_{S^2} \psi &= -l(l+1) \psi, \\
    \mathfrak{C}  \psi &= m(m+1) \psi .
\end{align}
In other words, the decomposition~\eqref{eqn:L2decomp} diagonalises the operators $\Delta_{S^2}$ and $\mathfrak{C}$ jointly.
\end{lemma}

\begin{proof}
We begin the proof by recalling some elementary facts about the representation theory of $\mathrm{so}(3)$, see~\cite{halllie} and~\cite[Chapter 17]{hall2}. For every $k \in \mathbb{N}_0$, there exists an irreducible representation of $\mathrm{so}(3)$ of dimension $k$ which is unique up to isomorphism. Furthermore, this representation comes from a representation of the Lie group $\text{SO}(3)$ if and only if $k$ is odd. The Casimir element of a representation $\rho$ of $\mathrm{so}(3)$ is defined by $C_{\rho}=- \sum_{i=1}^3 \rho(F_i)^2$, where $F_i, i=1,2,3$ are a set of generators of $\mathrm{so}(3)$. The definition is independent of the choice of generators. Note our sign convention, which corresponds with the one typically used in the physics literature. If $\rho$ is irreducible, the Casimir element $C_\rho$ is diagonal, $C_{\rho}  = l(l+1) \, \id$, where $k=2l+1$.

Consider the unitary representation $\Pi$ of $\mathrm{SO}(3)$ on $L^2(S^2 ; \C^3)$ given by
\begin{equation} \label{eq:liegrouprep}
\begin{gathered}
    \Pi: \mathrm{SO}(3) \rightarrow U(L^2(S^2;\C^3)) \\
    R \mapsto (\Psi(y) \mapsto R \Psi(R^T y)),
\end{gathered}
\end{equation}
and its associated representation $\pi$ of the Lie algebra $\mathrm{so}(3)$. We aim to compute the Casimir operator $C_\pi$ and to find a decomposition of $L^2(S^2 ; \C^3)$ into irreducible subspaces. Since the representation of $\mathrm{so}(3)$  comes from the representation~\eqref{eq:liegrouprep} of $\mathrm{SO}(3)$, the resulting subspaces will have odd dimensions, see~\cite[Theorem 17.10]{hall2}. The Casimir operator is then guaranteed to be diagonal on each of these subspaces.

We begin by noting that under the standard identification of $L^2(S^2; \C^3)$ with $L^2(S^2;\C) \otimes \C^3$, the representation $\Pi$ may be identified with $\Pi = \Pi_1 \otimes \Pi_2$, where $\Pi_1$ and $\Pi_2$, are given by
\begin{gather}
\Pi_1: \mathrm{SO}(3) \rightarrow U(L^2(S^2;\C)) \\
R \mapsto (\psi(y) \mapsto \psi(R^T y)),
\end{gather}
respectively
\begin{gather}
\Pi_2: \mathrm{SO}(3) \rightarrow U(\C^3) \\
R \mapsto (z \mapsto R z).
\end{gather}
We denote the associated representations of the Lie algebra $\mathrm{so}(3)$ by $\pi_1$ respectively $\pi_2$. We note that the space $L^2(S^2; \C)$ may be decomposed into orthogonal irreducible subspaces spanned by spherical harmonics, $L^2(S^2 ; \C) = \bigoplus_{l=0}^{\infty} H_l$, where $H_l$ is spanned by the spherical harmonics of degree $l$. The Casimir operator of $\pi_1$ is then $C_{\pi_1}=- \Delta_{S^2}$ and $\Delta_{S^2} \Psi = - l(l+1) \Psi$ for all $\Psi \in H_l$. For $\Pi_2$, a quick computation shows that for any $F \in \mathrm{so}(3)$, $\pi_2(F)$ acts by multiplication by the matrix $F$ on $\C^3$. It immediately follows that $\pi_2$ is an irreducible representation of $\mathrm{so}(3)$ with dimension $3$, so that the Casimir operator satisfies $C_{\pi_2} z = 2 z$ for all $z \in \C^3$.

In order to compute the Casimir operator $C_\pi$ we note $\pi= \pi_1 \otimes \id + \id \otimes \pi_2$, which follows readily from $\Pi = \Pi_1 \otimes \Pi_2$. Therefore, 
\begin{equation}
C_{\pi}= C_{\pi_1} \otimes \id + \id \otimes C_{\pi_2} - 2 \sum_{i=1}^3 \pi_1(F_i) \otimes \pi_2(F_i),
\end{equation}
where $F_i, i=1,2,3$ are generators of $\mathrm{so}(3)$, given by
\begin{equation}
    F_i = \sum_{j,k=1}^3 \epsilon_{ijk} E_{j,k}.
\end{equation}
We therefore find that
\begin{equation}
    \pi_1(F_i) = \sum_{j,k=1}^3 \epsilon_{ijk} x_j \partial_k, \quad \pi_2(F_i) = F_i.
\end{equation}
A standard computation reveals that
\begin{align} \label{equ::casimir}
C_{\pi} &= - \Delta_{S^2} +2 +2  \mathfrak{M} = \mathfrak{C},
\end{align}
where $\mathfrak{M}$ is as in the statement of the Lemma.

In order to compute the desired decomposition of $L^2(S^2; \C^3)$ into irreducible subspaces, we use the decomposition of $\pi_1$ above to find $L^2(S^2; \C^3) = \bigoplus_{l=0}^{\infty} H_l \otimes \C^3$. Note that now $(\pi, H_l \otimes \C^3)$ is a tensor product of an irreducible representation of dimension $2l+1$ and of dimension $3$. By~\cite[Theorem C.1]{halllie} we therefore find that when $l>0$, we may decompose $H_l \otimes \C^3 =\bigoplus_{m=l-1}^{l+1} W_{l,m}$, a direct sum of three subspaces of $H_l \otimes \C^3$, each of which is irreducible under $\pi$ and such that $W_{l,m}$ has dimension $2m+1$. The subspace $W_{0,1} := H_0 \otimes \C^3$, which consists of the component-wise constant functions, is evidently irreducible for $\pi$. From the remark at the beginning of the proof it now follows that the Casimir operator $C_{\pi}$ acts as $C_{\pi}|_{W_{l,m}} = m(m+1) \, \id$. Furthermore, since each $W_{l,m} \subset H_l$, it also follows that $\Delta_{S^2}|_{W_{l,m}} = - l(l+1) \, \id$. This concludes the proof.
\end{proof}

In the proof of Lemma~\ref{lem:decomp}, borrowing notation from the proof, we established that $H_l \otimes \C^3$ is irreducible for the representation $\pi$ when $l=0$, and may be decomposed into a direct sum of irreducible subspaces when $l\geq 1$ as  $H_l \otimes \C^3 = \bigoplus_{|l-m| \leq 1} W_{l,m}$. Recall that $W_{0,1} := H_0 \otimes \C^3$ corresponds simply to the space of constant functions, so that a basis is given by the constant unit vectors. We will denote this basis by $Z^k_{0,1}(y) = e_k$. For $l \geq 1$, the direct sum decomposition can be made explicit by means of a Clebsch--Gordan basis. We compute this basis explicitly for the first few values of $l$ and $m$. This allows us to understand in which subspaces the mode solutions generated by symmetries lie. We use the notation $Z_{l,m}^k$, $l\geq 1, m \in \{ l-1,l,l+1 \}$ and $k=1, \dots 2m+1$ to denote the Clebsch--Gordan basis for the space $H_l \otimes \C^3 = W_{l,l-1} \oplus W_{l,l} \oplus W_{l,l+1}$. For fixed $l,m$, the functions $Z_{l,m}^k$, $k=1, \dots 2m+1$ form a basis of the subspace $W_{l,m}$. The basis functions may be computed by consulting a table of Clebsch--Gordan coefficients~\cite[Chapter 44]{clebschgordan}. Expressed in $(\tau,y)$ coordinates, the basis functions $Z^k_{l,m}$ are provided in the right-hand column of Table~\ref{table:clebschgordan}.

\newcommand{\tablespacing}{0.2cm}
\newcommand{\tablespacingbig}{0.6cm}
\renewcommand{\arraystretch}{1.1}
\begin{table}[ht]
	\caption{Clebsch--Gordan basis for the first few values of $(l,m)$. The basis functions were computed by consulting a table of Clebsch--Gordan coefficients~\cite{clebschgordan}. The functions are expressed in $(\tau, y) = (\tau, y_1, y_2, y_3)$ coordinates and $e_1, e_2, e_3$ denote the standard Euclidean basis vectors.}\label{table:clebschgordan}%
	\begin{tabular}{@{}lc@{}}
		\toprule
		$(l,m)$ & Clebsch--Gordan basis $Z^k_{l,m}$ \\
		\midrule
		$(0,1)$ & $e_1,e_2,e_3$ \\[\tablespacing]
            $(1,0)$ & $\frac{1}{|y|} y$ \\[\tablespacing]
            $(1,1)$ & $\frac{1}{|y|} \left( \begin{matrix} 0 \\ -y_3 \\ y_2 \end{matrix} \right),\frac{1}{|y|} \left( \begin{matrix} y_3 \\ 0 \\ -y_1 \end{matrix} \right),\frac{1}{|y|} \left( \begin{matrix} -y_2 \\ y_1 \\ 0 \end{matrix} \right)$  \\[\tablespacingbig]
            $(1,2)$ & $\frac{1}{|y|} \left( \begin{matrix} 0 \\ y_2 \\ -y_3 \end{matrix} \right),\frac{1}{|y|} \left( \begin{matrix} 0 \\ y_3 \\ y_2 \end{matrix} \right),\frac{1}{|y|} \left( \begin{matrix} y_3 \\ 0 \\ y_1 \end{matrix} \right), \frac{1}{|y|} \left( \begin{matrix} y_1 \\ -y_2 \\ 0 \end{matrix} \right) ,\frac{1}{|y|} \left( \begin{matrix} y_2 \\ y_1 \\ 0 \end{matrix} \right)$  \\[\tablespacingbig]
            $(2,1)$ & $\frac{1}{|y|^2} \left( \begin{matrix} -2y_1^2+y_2^2+y_3^2 \\ -3y_1 y_2 \\ -3 y_1 y_3 \end{matrix} \right), \frac{1}{|y|^2} \left( \begin{matrix} -3y_1 y_2 \\ y_1^2 -2y_2^2 + y_3^2 \\ -3 y_2 y_3 \end{matrix} \right), \frac{1}{|y|^2} \left( \begin{matrix} -3 y_1 y_3 \\ -3y_2 y_3 \\ y_1^2 + y_2^2 - 2 y_3^2 \end{matrix} \right) $ \\
		\botrule
	\end{tabular}
\end{table}
\renewcommand{\arraystretch}{1.0}

By comparison of Table~\ref{table:clebschgordan} and the mode solutions from Lemma~\ref{sec:modesolutionslemma}, we can identify in which subspace each of the solutions generated by symmetries lies. If we assume the convention that the functions in Table~\ref{table:clebschgordan} are listed such that in each row, the Clebsch--Gordan basis $Z^k_{l,m}$ is listed in order of increasing index $k$, and using the notation from Lemma~\ref{sec:modesolutionslemma}, then we find the following relationship
\begin{equation}
\Psi_{l,m}^k(y)  = f_{l,m}(r) \, Z^k_{l,m}(\theta,\phi),
\end{equation}
for $(l,m) \in \{(0,1),(1,0),(1,1),(2,1)\}$ and in addition,
\begin{equation}
\Phi_{0,1}^k(y) = g_{0,1}(r) \, Z^k_{0,1}(\theta,\phi),
\end{equation}
where we have defined the radial functions
\begin{gather}
f_{0,1}(r)=1, \quad g_{0,1}(r)=r^2-3, \quad f_{1,0}(r)=r, \quad f_{1,1}(r) =r, \quad  f_{2,1}(r)=r^2.
\end{gather}
We conclude that the subspaces $W_{1,1}$ and $W_{2,1}$ contribute to the mode space of the growth rate $\lambda=0$, while $W_{1,0}$ contributes to the solution space for the growth rate $\lambda=1$ and $W_{0,1}$ contributes to both.

\subsection{Decoupling the linearised equation} \label{sec:decoupling}
Using the basis $Z^k_{l,m}$ of $L^2(S^2; \C^3)$ obtained in Section~\ref{sec:clebschgordan}, any $\Psi \in C^{\infty}(\overline{B}_1(0); \R^3)$ may be decomposed as
\begin{equation} \label{eqn:basis_decomp}
\Psi(r,\theta,\phi) = \sum_{l \geq 1, |l-m| \, \leq 1} \sum_{k=1}^{2m+1} f^k_{l,m}(r) \, Z^k_{l,m}(\theta,\phi) + \sum_{k=1}^3 f^k_{0,1}(r) Z^k_{0,1}(\theta,\phi),
\end{equation}
where the sum is understood to converge in the Hilbert space $L^2(\overline{B}_1(0); \C^3)$, the radial functions are $f^k_{l,m} \in L^2([0,1] ; \C)$ and we use standard spherical coordinates $(r, \theta, \phi) \in [0,1] \times [0, \pi) \times [0, 2\pi)$ to represent a point $y \in \overline{B}_1(0)$. The following Lemma shows that this basis allows us to decouple the linearised wave maps equation.

\begin{lemma} \label{lem:decoupling}
Let $\Psi \in C^{\infty}(\overline{B}_1(0); \R^3)$ be a smooth mode solution of the wave maps equation with growth rate $\lambda \in \C$. Let $f^k_{l,m}: [0,1] \rightarrow \C$ and $f^k_0: [0,1] \rightarrow \C$ be the radial functions in the decomposition~\eqref{eqn:basis_decomp}. Then $f^k_{l,m} \in C^\infty([0,1]; \R)$ and for each $f^k_{l,m}$ the function $\varphi^k_{l,m} = \frac{1}{1+r^2} f^k_{l,m}$ solves the following ordinary differential equation
\begin{equation} \label{equ::modeequation}
(1-r^2) \partial_{rr} \varphi^k_{l,m} + \left( \frac{2}{r} - 2 (\lambda +1) r \right) \partial_r \varphi^k_{l,m} - (\lambda^2 + \lambda + V_{l,m}) \varphi^k_{l,m} = 0,
\end{equation}
where the potential $V_{l,m}$ is given by
\begin{equation} \label{eqn:vlm}
V_{l,m} = \frac{(4+2m(m+1)-l(l+1)) r^4 + (2m(m+1)-12) r^2 +l(l+1)}{r^2(1+r^2)^2}.
\end{equation}
\end{lemma}

\begin{proof}
We begin by writing out the linearised wave maps equation~\eqref{equ::linintro} around the blow-up solution $u_0$ explicitly in self-similar coordinates $(\tau, y)$. A quick computation shows that the Christoffel symbols take the form
\begin{equation}
\Gamma^i_{jk}(u) = \frac{2}{1+ \left| u \right|^2}(u_i \delta_{jk} -u_j \delta_{ik} -u_k \delta_{ij}).
\end{equation}
A change of coordinates and elementary computations yield
\begin{equation}
\nabla \Gamma(y) \cdot \Psi \: Q_0(y,y) = \frac{2(1-r^2)}{1+r^2} \Psi,
\end{equation}
as well as the identity
\begin{equation}
2 \Gamma(y) Q_0(y,\Psi) = -\frac{4}{1+r^2} \mathfrak{M} \Psi + \frac{4 r^3- 4 r}{1+r^2} \partial_r \Psi + \frac{4r^2}{1+r^2} \partial_{\tau} \Psi,
\end{equation}
where $\mathfrak{M}$ is as in Lemma~\ref{lem:decomp}. Recall we showed the identity $\mathfrak{M} = \frac{1}{2}(\mathfrak{C} + \Delta_{S^2} - 2)$ in Lemma~\ref{lem:decomp}. By expressing the free wave operator $\Box \Psi$ in self-similar coordinates we find
\begin{equation} \label{linearizedeqwithcasimir}
\begin{aligned}
-\partial_{\tau \tau} \Psi -\frac{1-3r^2}{1+r^2} \partial_{\tau} \Psi &-2 r \partial_{r \tau} \Psi +(1-r^2) \partial_{rr} \Psi + 2 \frac{(1-r^2)^2}{r(1+r^2)} \partial_r \Psi \\
 &-\frac{2}{1+r^2} \mathfrak{C} \Psi + \frac{1-r^2}{r^2(1+r^2)} \Delta_{S^2} \Psi + \frac{6-2r^2}{1+r^2} \Psi =0.
\end{aligned}
\end{equation}
Let us now decompose $\Psi$ in the space $L^2(S^2; \C^3)$ as in equation~\eqref{equ::modeequation}. Since the $Z^k_{l,m}$ are mutually orthogonal in $L^2(S^2; \C^3)$, the radial coefficients can be computed up to normalisation constants as
\begin{equation}
f_{l,m}^k(r) = \int_{S^2} \Psi(r \omega) Z_{l,m}^k(\omega) \, d \omega.
\end{equation}
This identity and $\Psi \in C^{\infty}(\overline{B}_1(0); \R^3)$ together with the fact that all basis functions $Z^k_{l,m}$ may be chosen to be real-valued implies that $f_{l,m}^k \in C^\infty([0,1], \R)$. Lemma~\ref{lem:decomp} implies that $-\Delta_{S^2} Z_{l,m}^k = l(l+1) Z_{l,m}^k$ and $\mathfrak{C} Z_{l,m}^k = m(m+1) Z_{l,m}^k$. Inserting these two identities in equation~\eqref{linearizedeqwithcasimir}, we find that $f^k_{l,m}$ solves the equation
\begin{equation}
\begin{aligned}
- 2r \partial_{\tau r} f^k_{l,m} &+ (1-r^2) \partial_{rr} f^k_{l,m} + 2 \frac{(1-r^2)^2}{r(1+r^2)} \partial_r f^k_{l,m}  \\
&+ \frac{-2r^4+ C_{l,m} r^2 - l(l+1)}{r^2(1+r^2)} f^k_{l,m} -\lambda^2 f^k_{l,m} - \frac{1-3r^2}{1+r^2} \lambda f^k_{l,m} =0,
\end{aligned}
\end{equation}
where the coefficient $C_{l,m}$ is given by
\begin{equation}
C_{l,m}=6 - 2m(m+1)+l(l+1).
\end{equation}
This equation can be simplified through the following transformation of the dependent variable: $f^k_{l,m} = (1+r^2) \varphi^k_{l,m}$. Notice that $1+r^2$ and its inverse are regular at every point in $[0,1]$, so that the smoothness properties of solutions are unchanged by this. Then $\varphi^k_{l,m}$ solves
\begin{equation}
(1-r^2) \partial_{rr} \varphi^k_{l,m} + \left( \frac{2}{r} - 2 (\lambda +1) r \right) \partial_r \varphi^k_{l,m} - (\lambda^2 + \lambda + V_{l,m}) \varphi^k_{l,m} = 0,
\end{equation}
where the potential $V_{l,m}$ is now given by
\begin{align}
V_{l,m} &= - \frac{(C_{l,m} -10) r^4 + (C_{l,m} - l(l+1) +6) r^2 -l(l+1)}{r^2(1+r^2)^2}.
\end{align}
By inserting the expression for $C_{l,m}$ into the above identity, the reader may easily verify that $V_{l,m}$ is given by~\eqref{eqn:vlm}.
\end{proof}

\begin{remark} \label{rem:corotational_case}
We remark that if we set $l=1,m=0$, equation~\eqref{equ::modeequation} reduces to the co-rotational case. In fact, equation~\eqref{equ::modeequation} agrees with the equation studied in~\cite{donn}. We remark that the fact that we only obtain the equation in the form studied in~\cite{donn} after performing the transformation $f = (1+r^2) \varphi$ is natural under the light of the different conventions used: When expressed in our coordinates, Donninger et al. use the ansatz $\tan\left(\frac{f(r)}{2}\right) \frac{y}{r}$ whereas in our setting it is most natural to use $f(r) \frac{y}{r}$. In the linearisation, this induces a factor of $\tan'(\arctan(r))=1+r^2$.
\end{remark}

\section{Removing mode solutions generated by symmetries} \label{sec:removingsols}
We showed in Section~\ref{sec:symmetries_and_modes} that the symmetries of the wave maps equation generate mode solutions with growth rates $\lambda \in \{0,1\}$ to the linearised equation around $u_0$. In Section~\ref{sec:clebschgordan} we then showed that each of these mode solutions corresponds to a solution to equation~\eqref{equ::modeequation}. To show that there are no additional non-trivial smooth mode solutions with $\Re \lambda \geq 0$, we project the known solutions away using the `supersymmetry trick'~\cite[Section 3.5]{donn}.

\begin{definition}[SUSY transform] \label{def:susy}
Let $\{ \varphi_{0,1}^0, \varphi_{0,1}^1, \varphi_{1,0}, \varphi_{1,1}, \varphi_{2,1} \}$ be the set of solutions to equation~\eqref{equ::modeequation} generated by the symmetries of the wave maps equation with corresponding growth rates $\lambda_{l,m}$, see Table~\ref{table:lemma5}. Let $\varphi \in C^\infty([0,1]; \R)$ be a smooth solution to equation~\eqref{equ::modeequation} with growth rate $\lambda \in \C$ such that $\Re \lambda \geq 0$. Let us define a class of multiplication operators given by
\begin{equation}
    M_{a,b} (\varphi)(r) = r^a (1-r^2)^b \varphi(r),
\end{equation}
where $a,b \in \C$. For $(l,m) \in \{ (1,0),(2,1),(1,1) \}$ we then define the supersymmetric transform of $\varphi$ as
\begin{equation}
    S_{l,m}(\varphi) = M_{-1, 1 - \frac{\lambda}{2}} (\partial_r - \omega_{l,m}) M_{1, \frac{\lambda}{2}} \varphi,
\end{equation}
where the weights $\omega_{l,m}$ are defined by
\begin{equation}
    \omega_{l,m} = \partial_r \log \left( M_{1, \frac{\lambda_{l,m}}{2}} \varphi_{l,m} \right).
\end{equation}
For $(l,m) = (0,1)$, we define
\begin{equation}
    S_{0,1}(\varphi) = M_{-1, 1 - \frac{\lambda}{2}} (\partial_r - \omega_{0,1}^1) M_{0, 1} (\partial_r - \omega_{0,1}^0) M_{1, \frac{\lambda}{2}} \varphi,
\end{equation}
where
\begin{equation}
\begin{gathered}
    \omega_{0,1}^0 = \partial_r \log (M_{1, 0} \varphi_{0,1}^0), \\
    \omega_{0,1}^1 = \partial_r \log (M_{0,1} (\partial_r - \omega_{0,1}^0) M_{1,1} \varphi_{0,1}^1).
\end{gathered}
\end{equation}
\end{definition}

\begin{lemma} \label{lem:susytransform}
Let $(l,m) \in \{ (0,1), (1,0),(1,1),(2,1) \}$ and assume $\varphi \in C^\infty([0,1]; \R)$ is a smooth solution to equation~\eqref{equ::modeequation} with parameters $l,m$ and growth rate $\lambda \in \C$, $\Re \lambda \geq 0$. Then the supersymmetric transform $S_{l,m}(\varphi)$ is well-defined and $S_{l,m}(\varphi) \in C^\infty([0,1]; \C)$. Furthermore, each $\tilde{\varphi} = S_{l,m}(\varphi)$ solves the ordinary differential equation
\begin{equation} \label{equ::modeequationtrans}
(1-r^2) \partial_{rr} \tilde{\varphi} + \left( \frac{2}{r} - 2 (\lambda +1) r \right) \partial_r \tilde{\varphi} - (\lambda^2 + \lambda + \tilde{V}_{l,m}) \tilde{\varphi} = 0,
\end{equation}
where the potentials $\tilde{V}_{l,m}$ are given in Table~\ref{table:lemma5}. In addition, we can characterise the kernel of the supersymmetric transform, depending on the value of $(l,m)$. If $(l,m) \in \{ (1,1),(2,1) \}$ and $S_{l,m}(\varphi) = 0$ then for some $c \in \C$
\begin{equation}
\begin{cases}
\varphi=0 \quad &\text{if } \lambda \neq 0 \\
\varphi = c \varphi_{l,m} \quad &\text{if } \lambda = 0
\end{cases}
\end{equation}
Similarly, when $(l,m)=(1,0)$ and $S_{1,0}(\varphi)=0$, then for some $c \in \C$
\begin{equation}
\begin{cases}
\varphi=0 \quad &\text{if } \lambda \neq 1 \\
\varphi = c \varphi_{1,0} \quad &\text{if } \lambda = 1
\end{cases}
\end{equation}
and finally when $(l,m)=(0,1)$ and $S_{0,1}(\varphi)=0$ then for some $c \in \C$
\begin{equation}
\begin{cases}
\varphi=0 \quad &\text{if } \lambda \neq 0,1 \\
\varphi = c \varphi_{0,1}^0  \quad &\text{if } \lambda = 0 \\
\varphi = c \varphi_{0,1}^1 \quad &\text{if } \lambda = 1
\end{cases}
\end{equation}
\end{lemma}

\begin{remark} \label{rem:vlm}
Let us extend the definition of $\tilde{V}_{l,m}$ trivially by setting $\tilde{V}_{l,m} = V_{l,m}$ whenever $(l,m) \notin \{ (0,1), (1,0),(1,1),(2,1) \}$. Given the result of Lemma~\ref{lem:susytransform}, it then suffices to prove that any smooth solution to equation~\eqref{equ::modeequationtrans} is the trivial (zero) solution.
\end{remark}

\newcommand{\tablespacingTwo}{0.15cm}
\newcommand{\tablespacingTwoBigger}{0.25cm}
\renewcommand{\arraystretch}{1.1}
\begin{table}[ht]
    \caption{Overview of the solutions $\varphi_{l,m}$, the weights $\omega_{l,m}$ from Definition~\ref{def:susy} and the potentials $V_{l,m}$ and $\tilde{V}_{l,m}$ from Lemma~\ref{lem:susytransform}, alongside the growth rate $\lambda_{l,m}$. We note that to keep notation simple, what is called $\omega_{0,1}$ in the first row is referred to as $\omega_{0,1}^0$ elsewhere, similarly for the second row and $\varphi_{l,m}$.}\label{table:lemma5}%
	\begin{tabular}{@{}lccccc@{}}
		\toprule
		$(l,m)$ & $\lambda_{l,m}$ & $\varphi_{l,m}$ & $\omega_{l,m}$ & $V_{l,m}$ & $\tilde{V}_{l,m}$ \\
		\midrule
		\multirow{2}{*}{$(0,1)$} & $0$ & $\frac{(r^2-3)}{1+r^2}$ & $\frac{r^4+6 r^2-3}{r \left(r^4-2 r^2-3\right)}$ & \multirow{2}{*}{$\frac{8 (r^2-1)}{(1+r^2)^2}$} & \multirow{2}{*}{$\frac{6}{r^2}$} \\[\tablespacingTwo]
             & $1$ & $ \frac{1}{1+r^2}$ & $\frac{r^4-9 r^2+6}{r(1-r^2)(3-r^2)}$ &  &  \\[\tablespacingTwoBigger]
            $(1,0)$ & $1$ & $\frac{r}{1+r^2}$ & $\frac{r^4+3 r^2-2}{r (r^4-1)}$ & $\frac{2r^4 - 12r^2 + 2}{r^2 (1+r^2)^2}$ & $\frac{6-2r^2}{r^2(1+r^2)}$ \\[\tablespacingTwoBigger]
            $(1,1)$ & $0$ & $\frac{r}{1+r^2}$ & $\frac{2}{r(1+r^2)}$ & $\frac{6r^4 - 8r^2 + 2}{r^2 (1+r^2)^2}$ & $\frac{6-2r^4}{r^2(1+r^2)}$ \\[\tablespacingTwoBigger]
            $(2,1)$ & $0$ & $\frac{r^2}{1+r^2}$ & $\frac{3+r^2}{r(1+r^2)}$ & $\frac{2r^4 - 8r^2 + 6}{r^2 (1+r^2)^2}$ & $\frac{12}{r^2(1+r^2)}$ \\
		\botrule
	\end{tabular}
\end{table}
\renewcommand{\arraystretch}{1.0}

\begin{proof}[Proof of Lemma~\ref{lem:susytransform}]
We note that each of the solutions generated by the symmetries of the wave maps equation is non-vanishing and smooth in the interior $r \in (0,1)$. It follows immediately from Definition~\ref{def:susy} that $S_{l,m}(\varphi)(r)$ is well-defined and smooth for $r \in (0,1)$. We demonstrate that $S_{l,m}(\varphi)$ solves equation~\eqref{equ::modeequationtrans} when $r \in (0,1)$ for the case $(l,m)=(1,1)$ and leave the remainder to the reader, as the computations are similar and do not add additional clarity.

Let us therefore consider the case $(l,m)=(1,1)$ and show that $S_{1,1}(\varphi)$ solves the mode equation~\eqref{equ::modeequationtrans} on $r \in (0,1)$. If we define $\psi(r) = r (1-r^2)^{\frac{\lambda}{2}} \, \varphi(r)$, then $\psi$ solves
\begin{equation} \label{susyeq1}
- \partial_{rr} \psi + \frac{V_{1,1}(r)(1-r^2) +\lambda ( \lambda - 2)}{(1-r^2)^2} \psi = 0.
\end{equation}

In particular, when we transform the solution $\varphi_{1,1}$ in this way, we obtain a solution $\psi_{1,1}$ to~\eqref{susyeq1} with $\lambda = 0$. Put another way, if we evaluate the potential in~\eqref{susyeq1} for $\lambda = 0$ and define this as
\begin{equation}
\mathcal{V}_{1,1} = \frac{V_{1,1}(r)}{1-r^2},
\end{equation}
then $\psi_{1,1}$ solves $- \partial_{rr} \psi_{1,1} + \mathcal{V}_{1,1} \psi_{1,1} = 0$. We now define $\omega_{1,1} = \frac{\partial_r \psi_{1,1}}{\psi_{1,1}}$ and note that $\omega_{1,1}$ is smooth and well-defined on $(0,1]$, since $\psi_{1,1}$ is nowhere vanishing on $(0,1]$. A quick computation allows one to verify that $\omega_{1,1}$ is given explicitly by the expression provided in Table~\ref{table:lemma5}. Then the differential operator from above may be factorised as
\begin{equation} \label{equ:step1}
    - \partial_{rr} +\mathcal{V}_{1,1} = (- \partial_r - \omega_{1,1}) (\partial_r - \omega_{1,1}).
\end{equation}
This motivates us to define $\tilde{\psi} = (\partial_r - \omega_{1,1}) \psi$ and apply the operator $\partial_r - \omega_{1,1}$ to~\eqref{susyeq1} after multiplying by $(1-r^2)^2$. We obtain
\begin{equation}
(\partial_r - \omega_{1,1}) \left( (1-r^2)^2 (-\partial_r - \omega_{1,1}) \tilde{\psi} \right) = \lambda (2- \lambda) \tilde{\psi}.
\end{equation}
By expanding this out we arrive at the equation
\begin{equation} \label{equ:step2}
-(1-r^2)^2 \partial_{rr} \tilde{\psi} + 4 r(1-r^2) \partial_r \tilde{\psi} + (1-r^2) W_{1,1}(r) \tilde{\psi} =\lambda (2- \lambda) \tilde{\psi},
\end{equation}
where the potential $W_{1,1}$ is given by
\begin{equation}
W_{1,1}(r) = (1-r^2)(\omega_{1,1}^2(r)- \partial_r \omega_{1,1}(r)) + 4 r \omega_{1,1}(r).
\end{equation}
Finally, we apply another transformation given by $\tilde{\psi} = r (1-r^2)^{\frac{\lambda}{2} -1} \tilde{\varphi}$ and find
\begin{equation} \label{decoupledsusy}
(1-r^2) \partial_{rr} \tilde{\varphi} + \left( \frac{2}{r}- 2 (\lambda+1) r \right) \partial_r \tilde{\varphi} - \lambda( \lambda +1) \tilde{\varphi} - \tilde{V}_{1,1} \tilde{\varphi} = 0
\end{equation}
where the potential $\tilde{V}_{1,1}$ is given by
\begin{equation}
\tilde{V}_{1,1}(r) = W_{1,1}(r) -2,
\end{equation}
which can be easily shown to be equal to the expression provided in the statement of the Lemma. Stringing together the transformations in this argument produces the SUSY transformation $S_{1,1}$ and proves that $S_{1,1}(\varphi)$ solves equation~\eqref{equ::modeequationtrans}. The remaining cases are treated in a similar fashion, with the exception of the case $(l,m)=(0,1)$, where two mode solutions need to be projected away. This is simply accomplished by repeating the analogue of steps~\eqref{equ:step1} to~\eqref{equ:step2} twice, once for the growth rate $\lambda=0$ and once for the growth rate $\lambda=1$.

Next, we show that $S_{l,m}(\varphi)$ is in fact smooth at the endpoints $r=0$ and $r=1$. We will make extensive use of the Fuchs--Frobenius theory of ordinary differential equations in the complex plane, see for instance~\cite[Chapter 4]{teschl}. First we consider the (un-transformed) solution $\varphi$ to equation~\eqref{equ::modeequation} with parameters $(l,m)$.

\begin{claim*}
$\varphi(r )\simeq r^l$ as $r \rightarrow 0$ and $\varphi(r) \simeq (1-r)^n$ for some integer $n\geq 0$ as $r \rightarrow 1$.
\end{claim*}
\noindent
To see this, we compute the Frobenius indices of the untransformed mode equation~\eqref{equ::modeequation}. At $r=0$ we find $\{l,-(l+1) \}$ and at $r=1$ we have $\{0,2-\lambda-\lambda^2 \}$. Therefore at $r=0$ there exists a fundamental system of the form
\begin{align}
\varphi_1(r) &= r^l f(r) \\
\varphi_2(r) &= r^{-l-1} g(r) + C \log(r) \varphi_1(r)
\end{align}
with $f,g$ analytic in a neighbourhood of $r=0$ such that $f(0)=g(0)=1$ and a constant $C \in \C$ (which might be zero). Since we demand $\varphi$ to be smooth at $r=0$ we see that $\varphi(r) \simeq r^l$ as $r \rightarrow 0$. Now first assume that $2-\lambda-\lambda^2 = n \in \mathbb{Z}$. If $n \geq 0$, a fundamental system is given by
\begin{align}
\tilde{\varphi}_1(r) &= (1-r)^n \tilde{f}(r) \\
\tilde{\varphi}_2(r) &=  \tilde{g}(r) + C \log(1-r)\tilde{\varphi}_1(r)
\end{align}
where $\tilde{f}, \tilde{g}$ are analytic around $r=1$ with $\tilde{f}(1) = \tilde{g}(1)=1$ and $C \in \C$ might be zero except when $n=0$. If $n < 0$, a fundamental system is given by
\begin{align}
\tilde{\varphi}_1(r) &= \tilde{f}(r) \\
\tilde{\varphi}_2(r) &= (1-r)^n \tilde{g}(r) + C \log(1-r)\tilde{\varphi}_1(r)
\end{align}
If $2-\lambda-\lambda^2 \notin \mathbb{Z}$, a fundamental system is given by
\begin{align}
\tilde{\varphi}_1(r) &= \tilde{f}(r) \\
\tilde{\varphi}_2(r) &= (1-r)^{2-\lambda-\lambda^2} \tilde{g}(r)
\end{align}
In all cases it follows that $\varphi(r) \simeq 1$ as $r \rightarrow 1$, or better, i.e. $\varphi(r) \simeq (1-r)^n$ for some $n>0$. This concludes the proof of our claim. \\

Now we use the result of the claim to show that $S_{l,m}(\varphi)$ extends smoothly to $r=0$ and $r=1$. Let us first consider the case where $(l,m) \in \{ (1,0), (1,1), (2,1) \}$. From the explicit formulas in Table~\ref{table:lemma5} we find that
\begin{equation}
\begin{gathered}
    \omega_{1,0}(r) \simeq r^{-1}, \quad \omega_{1,1}(r) \simeq r^{-1}, \quad \omega_{2,1}(r) \simeq r^{-1} \quad \text{ as } r \rightarrow 0 \\
    \omega_{1,0}(r) \simeq(1-r)^{-1}, \quad \omega_{1,1}(r) \simeq 1, \quad \omega_{2,1}(r) \simeq 1 \quad \text{ as } r \rightarrow 1 \\
\end{gathered}
\end{equation}
By tracing all of the transformations and using the above claim, we find that $S_{l,m}(\varphi)(r) \simeq r^{l-1}$ as $r \rightarrow 0$ and $S_{l,m}(\varphi)(r) \simeq (1-r)^n$ for some integer $n \geq 0$ as $r \rightarrow 1$. Therefore $S_{l,m}(\varphi)(r)$ is smooth on $[0,1]$. In the case $(l,m)=(0,1)$, the explicit expressions in Table~\ref{table:lemma5} allow us to conclude
\begin{equation}
\begin{gathered}
    \omega^0_{0,1}(r) \simeq r^{-1}, \quad \omega^1_{0,1}(r) \simeq r^{-1} \quad \text{ as } r \rightarrow 0 \\
    \omega^0_{0,1}(r) \simeq 1, \quad \omega^1_{0,1}(r) \simeq (1-r)^{-1} \quad \text{ as } r \rightarrow 1 \\
\end{gathered}
\end{equation}
By tracing the transformations, we find that $S_{0,1}(\varphi)(r) \simeq (1-r)^n$ for some integer $n \geq 0$ as $r \rightarrow 1$ and $S_{0,1}(\varphi)(r) \simeq r^{-2} \varphi + \alpha r^{-1} \partial_r \varphi + \beta \partial^2_{rr} \varphi = O(r^{-2})$ as $r \rightarrow 0$, where $\alpha, \beta$ are smooth functions on $[0,1]$. In this case it is therefore not immediately apparent that $S_{0,1}(\varphi)$ is smooth at $r=0$. The Frobenius indices of the transformed equation at $r=0$ are computed to be $\{ 2,-3\}$. Therefore, even though our crude estimate only gives us $S_{0,1}(\varphi)(r) = O(r^{-2})$ as $r \rightarrow 0$, it excludes the possibility of our transformed solution to behave like $S_{0,1}(\varphi)(r) \simeq r^{-3}$ so that we must in fact have $S_{0,1}(\varphi)(r) \simeq r^2$ as $r \rightarrow 0$. Therefore, $S_{0,1}(\varphi)$ is also smooth on $[0,1]$.

Finally, we characterise the kernel of the SUSY transform $S_{l,m}$. First consider the case $(l,m)=(1,0)$. Then $S_{1,0}(\varphi)=0$ is equivalent to $(\partial_r - \omega_{1,0}) \psi = 0$, where $\psi = M_{1, \frac{\lambda}{2}} \varphi$ as above. This first-order ODE has the general solution $\psi = c \psi_{1,0}$, $c \in \C$. Since we assume that $\varphi$ is a solution to~\eqref{equ::modeequation}, $\psi$ solves the analogue of equation~\eqref{susyeq1} for the case $(l,m)=(1,0)$, which we can rewrite as
\begin{equation} \label{eqn:10kernel}
0=(- \partial_r - \omega_{1,0})\underbrace{( \partial_r - \omega_{1,0}) \psi}_{=0} = \frac{\lambda (2- \lambda) - \lambda_{1,0} (2- \lambda_{1,0})}{(1-r^2)^2} \psi
\end{equation}
Since $\lambda_{1,0} = 1$, and $\varphi$ is assumed to be a nontrivial solution, it follows from equality~\eqref{eqn:10kernel} that $\lambda = 1$. Therefore we find $\varphi(r) = c \varphi_{1,0}$ as claimed.

Next, let us characterise the kernel of the SUSY transform when $(l,m) = (1,1)$ and $(l,m) = (2,1)$. The argument proceeds analogously to above, but here $\lambda_{1,1} = \lambda_{2,1} = 0$, so that the analogue of equation~\eqref{eqn:10kernel} allows us to conclude that $\lambda \in \{0,2\}$. We obtain two solutions corresponding to the two different values of $\lambda$, namely
\begin{equation}
\varphi^0_{l,m}(r) = c \varphi_{l,m}(r), \quad \varphi^2_{l,m}(r) = c \frac{\varphi_{l,m}(r)}{1-r^2},
\end{equation}
where $\varphi^0_{l,m}$ solves equation~\eqref{equ::modeequation} with $\lambda=0$ whereas $\varphi^2_{l,m}$ solves~\eqref{equ::modeequation} with $\lambda=2$ and for some constant $c \in \C$. Explicitly, these solutions are given by (up to constants)
\begin{equation}
    \varphi^2_{1,1}(r) = \frac{r}{1-r^4}, \quad \varphi^2_{2,1}(r) =\frac{r^2}{1-r^4}.
\end{equation}
These solutions blow up like $(1-r)^{-1}$ around $r=1$, so that they are not smooth on $[0,1]$. Since our solution is assumed to be smooth, we may exclude $\varphi^2_{1,1}$ and $\varphi^2_{2,1}$ and conclude that $\lambda = 0$ and $\varphi(r) = c \varphi_{l,m}(r)$ as claimed. In addition, we remark that the solutions $\varphi^2_{1,1}$ and $\varphi^2_{2,1}$ are not contained in the Sobolev space $H^s(0,1)$ when $s > \frac{3}{2}$, so that it seems reasonable to expect that they will not play a role even in an analysis of the linear stability of the blow-up solution.

Finally, we consider the kernel of the SUSY transform $S_{0,1}$. Let us introduce for brevity the shorthand $\bar{\psi} = M_{0, 1} (\partial_r - \omega_{0,1}^0) M_{1, \frac{\lambda}{2}} \varphi$ and $\psi = M_{1, \frac{\lambda}{2}} \varphi$. We distinguish the case where $\bar{\psi} = 0$ and $\bar{\psi} \neq 0$. If we assume that $\bar{\psi} = 0$ it follows that $(\partial_r - \omega_{0,1}^0) \psi = 0$. Therefore as above, we find that $\psi = c \psi_0$ for some $c \in \C$ and we can rewrite the equation that $\psi$ solves as
\begin{equation}
0=(- \partial_r - \omega_{0,1}^0)\underbrace{( \partial_r - \omega_{0,1}^0) \psi}_{=0} = \frac{\lambda (2- \lambda)}{(1-r^2)^2} \psi.
\end{equation}
Similarly to above, we then conclude that since $\varphi \neq 0$, it follows that $\lambda \in \{0,2\}$. Therefore we again obtain two solutions of equation~\eqref{equ::modeequation} with $(l,m)=(0,1)$ given by
\begin{equation}
\varphi_0(r) = c \varphi_{0,1}^0(r), \quad \varphi_2(r) = c \frac{\varphi_{0,1}^0(r)}{1-r^2},
\end{equation}
for some $c \in \C$ and where $\varphi_0$ solves equation~\eqref{equ::modeequation} with $\lambda=0$ and $\varphi_2$ solves equation~\eqref{equ::modeequation} with $\lambda=2$. Since we demand $\varphi$ to be smooth, we can exclude the solution $\varphi_2$ and conclude that $\lambda=0$ and $\varphi = c \varphi_{0,1}^0(r)$ in this case.

If we assume $\bar{\psi} \neq 0$, then from $(\partial_r - \omega_{0,1}^1) \bar{\psi} = 0$ it follows that $\bar{\psi} = c \bar{\psi}_1$, where we have defined $\bar{\psi}_i = M_{0, 1} (\partial_r - \omega_{0,1}^0) M_{1, \frac{\lambda}{2}} \varphi_{0,1}^i$, $i =0,1$. This is equivalent to
\begin{equation}
    (\partial_r - \omega_{0,1}^0) \psi = c (\partial_r - \omega_{0,1}^0) \psi_1.
\end{equation}
The  general solution of this first order ODE is given by $\psi=a \psi_0 + b \psi_1$ for $a,b \in \C$. We must have that $b \neq 0$, since otherwise $\bar{\psi} = 0$. As above, we can express the equation that $\bar{\psi}$ solves as
\begin{equation}
0=(- \partial_r - \omega_{0,1}^1)\underbrace{( \partial_r - \omega_{0,1}^1) \bar{\psi}}_{=0} = \frac{\lambda (2- \lambda)-1}{(1-r^2)^2} \bar{\psi}.
\end{equation}
Since we assume that $\bar{\psi} \neq 0$, it follows that $(1-\lambda)^2 = 0$ so that $\lambda=1$. Therefore we can compute the general solution $\varphi$ in this case as
\begin{equation}
\varphi(r) = a \frac{r^2-3}{(1+r^2)\sqrt{1-r^2}} + b \frac{1}{1+r^2}
\end{equation}
Since we demand $\varphi$ to be smooth, we may conclude $a=0$ and $\varphi = c \varphi_{0,1}^1$ and $\lambda=1$ as claimed. This concludes the proof.
\end{proof}

\section{Heun and hypergeometric standard form} \label{sec:heunform}
After having removed the solutions generated by the symmetries if necessary, we now want to transform equation~\eqref{equ::modeequationtrans} into standard form. For the case $(l,m)=(0,1)$ we find that equation~\eqref{equ::modeequationtrans} can be transformed to yield a hypergeometric differential equation, whereas for all other cases, we obtain a Heun equation. We begin by briefly recalling some details about the hypergeometric and Heun standard form and refer the reader to~\cite[Chapters 15 and 31]{handbook} for a comprehensive overview of the hypergeometric and Heun equations.

Heun's equation in canonical form is a second-order linear differential equation in the complex plane given by
\begin{equation} \label{heunstandard}
\frac {d^2u}{dz^2} + \left[\frac{\gamma}{z}+ \frac{\delta}{z-1} + \frac{\epsilon}{z-a} \right] \frac {du}{dz} + \frac {\alpha \beta z -q} {z(z-1)(z-a)} u = 0,
\end{equation}
where we assume $\alpha, \beta, \gamma, \delta, a, q \in \C$ are such that $|a| \, \geq 1$ and $\epsilon=\alpha+\beta-\gamma-\delta+1$ to ensure that the point $z= \infty$ is a regular singular point. This equation has regular singularities at $z=0,1,a,\infty$. Any linear second-order equation in the complex plane with four regular singular points may be brought to this form through a M\"{o}bius transformation of the independent variable and s-homotopic transformations of the dependent variable. The Frobenius indices of this equation at the points $z=0,1,a,\infty$ are $\{0,1-\gamma \},\{0,1-\delta \},\{ 0,1-\eps \},\{ \alpha,\beta \}$ respectively. The simpler relative of the Heun's equation with only three regular singular points at $z=0,1,\infty$ is the hypergeometric differential equation which in canonical form reads
\begin{equation} \label{hyperstandard}
\frac {d^2u}{dz^2} + \left( \frac{c}{z} + \frac{1+a+b-c}{z-1} \right) \frac {du}{dz} + \frac{ab}{z(z-1)} u = 0,
\end{equation}
where $a,b,c \in \C$. Similarly to Heun's equation, any linear second-order ODE in the complex plane with three regular singular points can be brought to this form using the same transformations. The Frobenius indices at the regular singular points $z=0,1,\infty$ are given by $\{ 0, 1-c \},\{ 0, c-a-b \},\{ a,b \}$ respectively. 

Recall from Remark~\ref{rem:vlm} that we trivially extend the definition of the transformed potentials $\tilde{V}_{l,m}$ as $\tilde{V}_{l,m} = V_{l,m}$ whenever $(l,m) \notin \{ (0,1), (1,0),(1,1),(2,1) \}$. We now describe how to bring equation~\eqref{equ::modeequationtrans} into standard form. We distinguish the two cases $(l,m)=(0,1)$ and $(l,m)$ with $l>0$.

\subsection{Case 1: \texorpdfstring{$(l,m)=(0,1)$}{(l,m)=(0,1)}.} An inspection of the explicit form of the potential $\tilde{V}_{0,1}$ reveals that in this case, equation~\eqref{equ::modeequationtrans} is a linear second-order equation which has four regular singular points at $r=0,\pm 1, \infty$. We introduce the variable $z=r^2$ and find equation~\eqref{equ::modeequationtrans} transforms to
\begin{equation} \label{}
\frac {d^2 \varphi}{dz^2} + \left( \frac{3}{2z}+ \frac{\lambda}{z-1} \right) \frac {d \varphi}{dz} + \frac{\lambda(\lambda+1) + \tilde{V}_{0,1}(z)}{4z(z-1)} \varphi = 0.
\end{equation}
This equation possesses only three regular singular points at $z=0,1,\infty$. The general principle of bringing a linear second-order equation with three regular singular points into standard hypergeometric form makes use of the fact that, if $z=a$ is a singular point with Frobenius indices $\{ \alpha, \beta \}$, then the equation solved by $\psi(z) = (z-a)^{\gamma} \varphi(z)$ has indices $\{ \alpha + \gamma, \beta + \gamma \}$ at the point $z=a$. Choosing $\gamma \in \{ - \alpha, -\beta \}$ and doing this for every finite singular point brings the equation into standard form. In our case, we make the ansatz $\varphi(z) = z \psi(z)$ and obtain the equation
\begin{equation} \label{equ::modeequnhyper}
\frac {d^2 \psi}{dz^2} + \left( \frac{7}{2z} +\frac{\lambda }{z-1} \right) \frac {d \psi}{dz} + \frac{\lambda ^2+5 \lambda +6}{4 (z-1) z} \psi = 0,
\end{equation}
which is in hypergeometric standard form with coefficients $a=\frac{3+\lambda}{2}, b=\frac{2+\lambda}{2}, c = \frac{7}{2}$. We remark that trivially, if $\varphi: [0,1] \rightarrow \C$ is a smooth solution to equation~\eqref{equ::modeequationtrans}, then $\psi$ as defined above is a solution to~\eqref{equ::modeequnhyper} which is smooth at $z=0, 1$. In addition, if $\psi = 0$ then necessarily also $\varphi = 0$.

\subsection{Case 2: \texorpdfstring{$(l,m)$}{(l,m)} with \texorpdfstring{$l>0$}{l>0}.} In this case, an inspection of the explicit form of the potentials $\tilde{V}_{l,m}$ reveals that equation~\eqref{equ::modeequationtrans} is a linear second-order equation with four regular singular points at $r=0,\pm 1, \pm i, \infty$. As above, we introduce the new variable $z=r^2$ and find
\begin{equation}
\frac {d^2 \varphi}{dz^2} + \left( \frac{3}{2z}+ \frac{\lambda}{z-1} \right) \frac {d \varphi}{dz} + \frac{\lambda(\lambda+1) + \tilde{V}_{l,m}(z)}{4z(z-1)} \varphi = 0.
\end{equation}
This equation now only has $4$ regular singular points at $z=0, \pm 1, \infty$. For technical reasons which will become apparent in Section~\ref{sec:proofs}, we need to apply a coordinate transformation so that the only singular points in the unit disk are the points $z=0$ and $z=1$. We therefore apply the M\"{o}bius transformation $z \mapsto \frac{2z}{1+z}$. Under this transformation, the singular points are mapped to $(0,1,-1,\infty) \mapsto (0,1,\infty,2)$. Note carefully that any solution that is smooth at $z=0$ and $z=1$ remains so after this transformation. Denoting the new variable again by $z$, we obtain the equation
\begin{equation}
\frac {d^2 \varphi}{dz^2} + \left( \frac{3}{2z}+ \frac{\frac{1}{2}-\lambda}{z-2} + \frac{\lambda}{z-1} \right) \frac {d \varphi}{dz} + \frac{\lambda ^2+\lambda +\tilde{V}_{l,m}\left(\frac{z}{2-z}\right)}{2 (z-2)^2 (z-1) z} \varphi = 0
\end{equation}
In order to bring this equation into canonical Heun form~\eqref{heunstandard}, we carry out the same procedure outlined in Case~1 above for each of the finite singular points. We make the ansatz $\varphi(z) = h_{l,m}(z) \psi(z)$, where $h_{l,m}(z)$ depends on $(l,m)$ and is given by
\begin{align}
    h_{1,0}(z) &= z(2-z)^{\frac{\lambda}{2}}, \\
    h_{1,1}(z) &= z (2-z)^{\frac{\lambda-1}{2}} , \\
    h_{2,1}(z) &= z^{\frac{3}{2}}(2-z)^{\frac{\lambda}{2}}
\end{align}
and in all remaining cases,
\begin{equation}
    h_{l,m}(z) = z^{\frac{l}{2}} (2-z)^{\frac{\lambda}{2}}.
\end{equation}
It is then straightforward to verify that the equation solved by $\psi$ is in standard Heun form, 
\begin{equation} \label{equ::modeequheun}
\frac {d^2 \psi}{dz^2} + p_{l,m} \frac {d \psi}{dz} + q_{l,m} \psi = 0,
\end{equation}
where the terms $p_{l,m},q_{l,m}$ are given by
\begin{equation} \label{equ::plm}
p_{l,m}=\begin{dcases}
\frac{7}{2z} +\frac{\lambda }{z-1}+\frac{1}{2(z-2)} &\text{if } (l,m)=(1,0) \\
\frac{7}{2 z}+\frac{\lambda }{z-1}-\frac{1}{2 (z-2)} &\text{if } (l,m)=(1,1) \\
\frac{9}{2z} +\frac{\lambda }{z-1}+\frac{1}{2(z-2)} &\text{if } (l,m)=(2,1) \\
\frac{2 l+3}{2 z}+\frac{\lambda }{z-1}+\frac{1}{2 (z-2)} &\text{other cases}
\end{dcases}
\end{equation}
and
\begin{equation}
q_{l,m}=\begin{dcases}
\frac{z \left(\lambda ^2+6 \lambda +8\right) -\lambda ^2-12 \lambda-12}{4 (z-2) (z-1) z} &\text{if } (l,m)=(1,0) \\
\frac{z (\lambda ^2+4 \lambda +3)-\lambda ^2-12 \lambda-7}{4 (z-2) (z-1) z} &\text{if } (l,m)=(1,1) \\
 \frac{z \left(\lambda ^2+8 \lambda +15\right) -\lambda ^2-16 \lambda -27}{4 (z-2) (z-1) z} &\text{if } (l,m)=(2,1)
\end{dcases}
\end{equation}
and in the remaining cases,
\begin{equation}
q_{l,m}=\frac{z \left((\lambda -2) (\lambda +4)+l^2+2 \lambda  l+2 l\right)-l^2-4 \lambda  l-2 l-2 m^2-2 m-\lambda^2-4\lambda+12}{4 (z-2) (z-1) z}.
\end{equation}
We remark that by construction, if $\varphi$ is a smooth solution to equation~\eqref{equ::modeequationtrans}, then $\psi$ is a solution to~\eqref{equ::modeequheun} which is regular at $z= 0$ and $z=1$. In addition, if $\psi=0$ it follows that $\varphi = 0$.

\section{Mode stability for \texorpdfstring{$(l,m) = (0,1)$}{(l,m)=(0,1)}} \label{sec:modestability_simplecase}
We have established in Section~\ref{sec:heunform} that for the case $(l,m)=(0,1)$ it is enough to show that the only solutions to equation~\eqref{equ::modeequnhyper} which are smooth at $z=0$ and $z=1$ are the trivial zero solution when $\Re \lambda \geq 0$. Since equation~\eqref{equ::modeequnhyper} is in hyper-geometric standard form, the argument simplifies significantly in this case.

\begin{proposition} \label{prop:simplecase}
Let $\Re \lambda \geq \modebound$ and let $\psi$ be a solution to equation~\eqref{equ::modeequnhyper} which is regular at $z=0$ and $z=1$. Then $\psi = 0$.
\end{proposition}
\begin{proof}
If $\psi$ is regular at $z=0$ and $z=1$ then it is holomorphic. The point $\infty$
is a regular singular point with Fredholm indices $a= \frac{3+\lambda}2$
and $ b= \frac{2+\lambda}2$ and hence 
\begin{equation}
    \left| \psi(z) \right| \le C (1+ |z|)^{-\frac{2+\Re \lambda}2}. 
\end{equation}
In particular $ \psi $ is an entire decaying function. By the maximum principle $\psi$ vanishes identically, which concludes the proof.
\end{proof}

%We begin by noting that the Frobenius indices at $z=0$ are $\{ 0,-\frac{5}{2} \}$ so that by the Fuchs-Frobenius theorem~\cite[Theorem 4.5]{teschl} there is a unique (up to multiplication by constants) solution smooth at $z=0$ given by the standard hypergeometric function
%\begin{equation}
%\psi(z) = {}_2 F_1 \left(\frac{3+\lambda}{2},\frac{2+\lambda}{2},\frac{7}{2};z \right),
%\end{equation}
%where ${}_2 F_1(a,b,c;z)$ denotes the Gauss hypergeometric series. We recall the definition of ${}_2 F_1(a,b,c;z)$, for $a,b,c \in \C$, $c \notin - \mathbb{N}_0$ as
%\begin{equation}
%{}_2 F_1(a,b,c;z) = \sum_{n=0}^{\infty} \frac{(a)_n (b)_n}{(c)_n n!} z^n,
%\end{equation}
%when $|z| \, < 1$ and by analytic continuation elsewhere. Here $(a)_n = a (a+1) \dots (a+n-1)$ denotes the Pochhammer symbol. Therefore it is clear that unless either $a \in - \mathbb{N}_0$ or $b \in - \mathbb{N}_0$, the coefficients of the series never vanish. In this case, we may compute the radius of convergence using the ratio test and find that the ratio of successive coefficients is given by
%\begin{equation}
%\frac{(a+n) (b+n)}{(c+n) (n+1)} \;\underset{n \rightarrow \infty}%{\longrightarrow} \; 1.
%\end{equation}
%Since in our case $a=\frac{3+\lambda}{2}$ and $b=\frac{2+\lambda}{2}$ and $\Re \lambda \geq \modebound$, we have $\Re a, \Re b \geq \frac{1}{2}$. Since $\psi(z)$ solves the hypergeometric differential equation in standard form, it can only be non-analytic at $z=0,1,\infty$. Therefore $\psi(z)$ cannot be smooth at $z=1$ unless $\psi = 0$.

\section{Mode stability for \texorpdfstring{$(l,m)$}{(l,m)} with \texorpdfstring{$l>0$}{l>0}}
 \label{sec:proofs}
 
We use the \textit{quasi-solution method} to show that there are no non-trivial smooth solutions to the equations~\eqref{equ::modeequationtrans} when $l>0$, $m \in \{ l-1,l,l+1 \}$ and $\Re \lambda \geq 0$. This section largely follows the presentation in~\cite{glogicthesis,glogic2,schoe}. The key result we will establish in this Section is the following Proposition:

\begin{proposition} \label{prop:heunsols}
	Let $\psi: \C \rightarrow \C$ be a solution to equation~\eqref{equ::modeequationtrans} with $l > 0$, $m \in \{ l-1,l,l+1 \}$ and with growth rate $\Re \lambda \geq 0$. If $\psi$ is smooth at $z = 0$ and $z=1$, then $\psi = 0$.
\end{proposition}

The remainder of this section is dedicated to the proof of Proposition~\ref{prop:heunsols}. An outline of the argument may be given as follows: We first note that $\psi$ solves a Heun equation~\eqref{heunstandard} in canonical form with regular singular points at $0, 1, 2, \infty$. Recall that the Frobenius indices at $z=0$ are given by $\{0, 1-\gamma\}$. For simplicity of notation, we will suppress the dependence of the Heun coefficients on $l,m$. It can be directly verified from equation~\eqref{equ::plm} that $1-\gamma \notin \mathbb{N}_0$ for all $l,m$ with $l>0$. Therefore there is a unique Fuchs--Frobenius solution which is smooth around the point $z=0$; let us denote this solution by $\psi(z)$. Then $\psi$ has the expansion
\begin{equation} \label{eqn:expansion}
	\psi(z) = \sum_{n=0}^{\infty} x_n z^n,
\end{equation}
which is convergent at least for $|z| \, < 1$. We aim to show that the convergence radius is indeed $1$, so that $\psi(z)$ cannot be smooth at both $z=0$ and $z=1$ unless $\psi = 0$. To this end, we define the ratio
\begin{equation}
	r_n = \frac{x_{n+1}}{x_n},
\end{equation}
so that $r=\lim_{n \rightarrow \infty} r_n$ is the inverse convergence radius of the series~\eqref{eqn:expansion}, or in other words, the series converges for $\left| z \right| < r^{-1}$. We use the inverse convergence radius for convenience here. In order to prove that $r = 1$, we use that $\psi$ solves the Heun equation in canonical form and hence the sequence $r_n$ satisfies a certain recurrence relation. Unfortunately, studying this recurrence relation directly is intractable. Instead, we construct an explicit approximate solution (or \emph{quasi-solution}) $\tilde{r}_n$ to this recurrence relation, which is then shown to be within a sufficiently small error interval around the true solution. It is then easily shown that the resulting \emph{quasi-solution} satisfies $\lim_{n \rightarrow \infty} \tilde{r}_n = 1$, and the error bound is sufficiently tight to imply that also $\lim_{n \rightarrow \infty} r_n =r = 1$.

We now derive the recurrence relation for $r_n$. First note that standard theory on the Heun equation implies that the coefficients $x_n \in \C$ of the expansion~\eqref{eqn:expansion} satisfy
\begin{equation}
	x_0 = 1, \quad a \gamma x_1 - q = 0,
\end{equation}
and for higher values of $n$ the following three term recurrence relation
\begin{equation} \label{equ::threetermrelation}
	R_n x_{n+1} - (Q_n + q) x_n + P_n x_{n-1} = 0, \quad n \geq 1,
\end{equation}
with coefficients given by 
\begin{gather}
	R_n = a (n+1) (\gamma +n),\\
	P_n = (n-1+\alpha ) (n-1+\beta), \\  
	Q_n = n ((n-1+\gamma)(1+a) +a\delta + \epsilon).
\end{gather}
Equation~\eqref{equ::threetermrelation} immediately implies that $r_n$ satisfies the recurrence relation
\begin{equation} \label{equ::recurrencern}
	r_0= \frac{q}{a \gamma} \quad \text{and} \quad r_{n} = A_n + \frac{B_n}{r_{n-1}}, \quad n \geq 1,
\end{equation}
with the coefficients
\begin{equation}
	A_n = \frac{Q_n+q}{R_n}, \quad B_n = - \frac{P_n}{R_n}
\end{equation}
These coefficients may be computed explicitly for each value of $(l,m)$ and we find
\begin{equation} \label{equ::recursioncoeffA}
	A_n = \begin{dcases*}
		\frac{\lambda ^2+12 \lambda +12 n^2+8 (\lambda +4) n+12}{4 \left(2 n^2+9 n+7\right)} \quad & $(l,m)=(1,0)$ \\
		\frac{\lambda ^2+12 \lambda +12 n^2+8 \lambda  n+28 n+7}{8 n^2+36 n+28} \quad & $(l,m)=(1,1)$ \\
		\frac{\lambda ^2+16 \lambda +12 n^2+8 \lambda  n+44 n+27}{8 n^2+44 n+36} \quad & $(l,m)=(2,1)$
	\end{dcases*}
\end{equation}
and in the remaining cases,
\begin{equation}
	A_n = \frac{\lambda ^2+4 \lambda +l^2+2 l (2 \lambda +6 n+1)+2 m^2+2 m+12 n^2+8 \lambda  n+8 n-12}{4 (n+1) (2 l+2 n+3)}.
\end{equation}
For the coefficients $B_n$ one finds
\begin{equation} \label{equ::recursioncoeffB}
	B_n = \begin{dcases*}
		-\frac{(\lambda +2 n) (\lambda +2 n+2)}{4 (n+1) (2 n+7)} \quad  &$(l,m)=(1,0)$  \\
		-\frac{(\lambda +2 n-1) (\lambda +2 n+1)}{4 (n+1) (2 n+7)} \quad &$(l,m)=(1,1)$  \\
		-\frac{(\lambda +2 n+1) (\lambda +2 n+3)}{4 (n+1) (2 n+9)} \quad &$(l,m)=(2,1)$  \\
		-\frac{(\lambda +l+2 n-4) (\lambda +l+2 n+2)}{4 (n+1) (2 l+2 n+3)}  \quad &\text{other values}
	\end{dcases*}.
\end{equation}
We note that the coefficients $A_n, B_n$ for $(l,m)=(1,0)$ agree with the coefficients arising in~\cite{glogic2}, as expected since the case $(l,m)=(1,0)$ represents the co-rotational case. In the remainder of this section, we will prove that $r_n \rightarrow 1$ for all $l,m$ under consideration.

\subsection{Restraining the set of possible convergence radii}
As a first step towards proving that $r=\lim_{n \rightarrow \infty} r_n =1$ , we first prove that the limit can only equal $1$ or $\frac{1}{2}$. Since $r$ is the inverse convergence radius, this statement amounts to the statement that the solution~\eqref{eqn:expansion} is either singular at $z=1$ (in case $r=1$), or regular at $z=1$ but singular at $z=2$ (in case $r=\frac{1}{2}$).

\begin{lemma} \label{lem:poincare}
Let $l>0$ and $m \in \{l-1,l,l+1\}$. Then the sequence $r_n$ defined in equation~\eqref{equ::recurrencern} is convergent and $\lim_{n \rightarrow \infty} r_n \in \{ \frac{1}{2},1 \}$.
\end{lemma}
\begin{proof}
The proof is essentially an application of Poincar\'{e}'s theorem on recurrence relations~\cite[Theorem 8.9]{recurrencerelations}. We recall the theorem here: Consider the $k$-th order linear difference equation with complex coefficients $p^i_n \in \C$ given by
\begin{equation} \label{equ::lineardiffequ}
	x_{n+k} + p^1_n \: x_{n+k-1} + \dots + p^k_n \: x_n = 0.
\end{equation}
We assume that for each coefficient, $\lim_{n \rightarrow \infty} p^i_n = p^i \in \R$ exists and is real. Then the characteristic equation associated to~\eqref{equ::lineardiffequ} is given by
\begin{equation}
	t^k + p_1 t^{k-1} + \dots + p_k = 0,
\end{equation}
and its roots $t_1, \dots, t_k \in \C$ are called the characteristic roots of equation~\eqref{equ::lineardiffequ}. Under these assumptions, the Poincar\'{e} recurrence theorem states that if $x_n$ is a solution to~\eqref{equ::lineardiffequ} and assume $|t_i| \, \neq |t_j|$ for $i \neq j$, then either
\begin{equation}
	\lim_{n \rightarrow \infty} \frac{x_{n+1}}{x_n} = t_i,
\end{equation}
for some characteristic root $t_i$ or $x_n = 0$ eventually. To show that the assumptions of the theorem hold, we first note that the limits of the coefficients of the recursion relation for all $l,m$ are given by
\begin{equation}
	\lim_{n \rightarrow \infty} A_n = \frac{3}{2}, \quad \lim_{n \rightarrow \infty} B_n = -\frac{1}{2},
\end{equation}
as may be observed by inspection of equations~\eqref{equ::recursioncoeffA} and~\eqref{equ::recursioncoeffB}. Therefore the characteristic equation of the recursion relation~\eqref{equ::recurrencern} is
\begin{equation}
	t^2-\frac{3}{2}t + \frac{1}{2} = 0.
\end{equation}
This equation has the two roots $t_0=1$ and $t_1=\frac{1}{2}$. Finally, we establish that there cannot exist a $N$ so that for $n \geq N$, $x_n=0$. Suppose there did. Since the coefficients in the recursion relation never vanish, we find that if $x_{n+1},x_{n}=0$ also $x_{n-1}=0$. By backwards induction we therefore conclude that $x_0=0$ which is a contradiction since $x_0=1$. Therefore we have a nowhere vanishing solution. Therefore, an application of the Poincar\'{e} recurrence theorem proves the claim.
\end{proof}

\subsection{Constructing the quasisolutions}

An exact closed-form solution to the recurrence relation~\eqref{equ::recurrencern} does not seem plausible to the authors, so that instead, we aim to construct an approximative solution or \emph{quasisolution} $\tilde{r}_n$ with the properties that
\begin{gather}
	\lim_{n \rightarrow \infty} \tilde{r}_n = 1, \label{ equ:quasisoln1} \\
	 \limsup_{n \rightarrow \infty} \Big| \, \frac{r_n}{\tilde{r}_n}-1 \, \Big| \, < \frac{1}{2 }. \label{equ:quasisoln2}
\end{gather}
Assuming the existence of such a quasisolution $\tilde{r}_n$, this would immediately imply $\lim_{n \rightarrow \infty} r_n = 1$ and therefore conclude the proof of Proposition~\ref{prop:heunsols}. We compute the quasisolutions by carefully considering the asymptotic behaviour of the sequence $r_n$ as $\lambda \rightarrow \infty$ and $l \rightarrow \infty$. The approximations we computed are given in Table~\ref{table_approximations}. Note that for the corotational case $(l,m)=(1,0)$ we use the approximation computed in~\cite{glogic2}. With the explicit form of $\tilde{r}_n$ given in Table~\ref{table_approximations}, it is a trivial exercise to compute the limit as $n \rightarrow \infty$.

\begin{lemma} \label{lem:quasisoln_limit}
For all $(l,m)$ with $l>0$ we have $\lim_{n \rightarrow \infty} \tilde{r}_n = 1$.
\end{lemma}

\begin{remark}
The details of how the quasisolutions listed in Table~\ref{table_approximations} are obtained are irrelevant to the proof. However, for the sake of completeness, we provide the reader with some details of how these approximations were obtained. For further details on the computation of quasi-solutions in general, we refer the reader to~\cite{glogic2}.

We begin by studying the behaviour of the sequence $r_n$ as $\lambda \rightarrow \infty$. We observe that $r_n$ is a rational function in $\lambda$ of order two, which leads us to guess that $r_n \sim_{n,l,m} \lambda^2$ in the limit as $\lambda \rightarrow \infty$. Indeed by dividing equation~(\ref{equ::recurrencern}) by $\lambda^2$ it is easy to see that $\lambda^{-2} r_n \sim \lambda^{-2} A_n$ as $\lambda \rightarrow \infty$. This provides the coefficient of $\lambda^2$ in the approximation. Similarly, one can obtain the term linear in $\lambda$. We repeat this procedure to obtain the behaviour of $r_n$ as $l \rightarrow \infty$.

In order to capture the behaviour as a function of $n$, we distinguish the three cases where $m=l-1, l, l+1$. We then set $\lambda=0$ and consider the lowest relevant value of $l$. We may now numerically compute the sequence $r_n$ for these special values of $\lambda$ and $l$ for $n=1, \dots 50$. The approximation is completed by fitting a rational function in $n$ with integer coefficients to this sequence. 

The choice of this fit is somewhat delicate in practice, since it interacts with the resulting error bounds. For this reason, we had to separate the case $(l,m) = (l,l)$, $l \geq 2$ into two subcases, namely $(l,m)=(2,2)$ and $(l,m)=(l,l)$ with $l \geq 3$ and find separate approximations for each case. Similarly, we had to divide the case $(l,m) = (l,l+1)$, $l \geq 1$ into the three subcases $(l,m) = (1,2)$, $(l,m) = (2,3)$ and $(l,m) = (l,l+1)$ with $l \geq 3$.
\end{remark}

\renewcommand{\arraystretch}{1.6}
\begin{table}[ht]
	\caption{Explicit form of the quasisolutions to recurrence relation~\eqref{equ::recurrencern}. The value of $(l,m)$ is provided in the first column. In the last three rows, we consider the cases $(l,m)$ where $l \geq 3$.}\label{table_approximations}%
	\begin{tabular}{@{}lc@{}}
		\toprule
		$(l,m)$ & Quasisolution $\tilde{r}_n$ \\
		\midrule
		$(1,0)$ & $\frac{\lambda^2}{8 n^2+36 n+28} + \frac{\lambda (2 n+3) }{2 n^2+9 n+7} + \frac{2 n+4}{2 n+7}$ \\
        $(1,1)$ & $\frac{\lambda ^2}{8 n^2+36 n+28}+\frac{\lambda  (2 n+3)}{2 n^2+9 n+7}+\frac{15 n+15}{15 n+40}$ \\
        $(1,2)$ & $\frac{\lambda ^2}{8 n^2+28 n+20}+\frac{\lambda  (2 n+2)}{2 n^2+7 n+5}+\frac{2 n+12}{2 n+14}$ \\
        $(2,1)$ & $\frac{\lambda ^2}{8 n^2+44 n+36}+\frac{\lambda  (2 n+4)}{2 n^2+11 n+9}+\frac{2 n+9}{2 n+12}$ \\
        $(2,2)$ & $\frac{\lambda ^2}{8 n^2+20 n+2 (8 n+8)+12}+\frac{\lambda  (2 n+3)}{2 n^2+5 n+2 (2 n+2)+3}+\frac{6 n+30}{6 n+35}$ \\
        $(2,3)$ & $\frac{\lambda ^2}{8 n^2+20 n+2 (8 n+8)+12}+\frac{\lambda  (2 n+3)}{2 n^2+5 n+2 (2 n+2)+3}+\frac{4 n+42}{4 n+47}$ \\
        %\midrule
        %& For the following $(l,m)$ we assume $l \geq 3$ \\
        %\midrule
        $(l \geq 3,l-1)$ & $\frac{\lambda ^2}{l (8 n+8)+8 n^2+20 n+12}+\frac{\lambda  (l+2 n+1)}{l (2 n+2)+2 n^2+5 n+3}+\frac{3 (l-3)}{8 n+8}+\frac{6 n+11}{6 n+20}$ \\
        $(l \geq 3,l)$ & $\frac{\lambda ^2}{l (8 n+8)+8 n^2+20 n+12}+\frac{\lambda  (l+2 n+1)}{l (2 n+2)+2 n^2+5 n+3}+\frac{3 (l-2)}{8 n+8}+\frac{n+4}{n+6}$ \\
        $(l \geq 3,l+1)$ & $\frac{\lambda ^2}{l (8 n+8)+8 n^2+20 n+12}+\frac{\lambda  (l+2 n+1)}{l (2 n+2)+2 n^2+5 n+3}+\frac{3 (l-1)}{8 n+8}+\frac{2 n+11}{2 n+15}$ \\
		\botrule
	\end{tabular}
\end{table}
\renewcommand{\arraystretch}{1.0}

\subsection{Bounding the relative error}
Having provided the explicit form of the quasisolutions $\tilde{r}_n$ and established that they converge to $1$, we now want to control the size of the relative error
\begin{equation}
	e_n = \frac{r_n}{\tilde{r}_n} - 1.
\end{equation}
First we note the following that if $r_n$ satisfies~(\ref{equ::recurrencern}) then for any choice of $\tilde{r}_n \in \C \setminus \{ 0 \}$ the relative error satisfies the recurrence relation
\begin{equation} \label{equ:anbn}
	e_{n} = a_n + b_n \frac{e_{n-1}}{1+e_{n-1}}, \qquad n \geq 1
\end{equation}
with
\begin{equation} \label{equ:anbn2}
	a_n = \frac{A_n \tilde{r}_{n-1}+B_n}{\tilde{r}_{n-1} \tilde{r}_{n}} -1, \quad b_n = - \frac{B_n}{\tilde{r}_{n-1} \tilde{r}_{n}}.
\end{equation}
The explicit form of the coefficients $a_n, b_n$ is quite complicated, so we only provide it in a digital format, see Appendix~\ref{sec::github}. We begin our analysis of the relative error by establishing upper bounds for the coefficients $a_n, b_n$.

\renewcommand{\arraystretch}{1.6}
\begin{table}[ht]
	\caption{Upper bounds for the coefficients $a_n$ and $b_n$ defined in equation~\eqref{equ:anbn}. The value of $(l,m)$ is provided in the first column. In the last row we assume $l \geq 3$ and in the antepenultimate and penultimate row we assume $l \geq 4$.}\label{table_bounds}%
	\begin{tabular}{@{}lllll@{}}
		\toprule
		$(l,m)$ & $\bar{a}_n$ & $\bar{b}_n$  & $n_0$ & $u$ \\
		\midrule
		$(1,1)$ & $\frac{72 + 125 n}{300 (-3 + 5 n)}$  & $\frac{-11 + 16 n}{4 (-1 + 8 n)}$   & $2$  & $\frac{3}{10}$  \\
		$(1,2)$ & $\frac{75 n+266}{150 (6 n+1)}$  & $\frac{25 n-11}{50 (n+1)}$    & $4$  & $\frac{1}{3}$  \\
		$(2,1)$ & $\frac{-71 + 100 n}{300 (-5 + 4 n)}$ & $\frac{-37 + 50 n}{25 (-1 + 4 n)}$   & $2$  & $\frac{3}{10}$  \\
		$(2,2)$ &  $\frac{125 n+482}{300 (5 n+2)}$ & $\frac{400 n-179}{100 (8 n+11)}$   & $3$  & $\frac{3}{10}$  \\
		$(2,3)$ & $\frac{5 n+12}{60 n}$ & $\frac{125 n-96}{50 (5 n+1)}$  & $2$ & $\frac{3}{10}$  \\
		$(3,2)$ & $\frac{125 n-121}{300 (5 n-7)}$ & $\frac{400 n-319}{100 (8 n-3)}$    & $2$ & $\frac{3}{10}$  \\
		$(3,3)$ & $\frac{800 n-443}{600 (16 n-19)}$ & $\frac{104 n-133}{8 (26 n-27)}$    & $3$ & $\frac{3}{10}$  \\
        $(l \geq 4,l-1)$ & $\frac{-1016 + 272 l + 125 n}{300 (-23 + 5 l + 5 n)}$ & $\frac{-63 + 11 l + 20 n}{20 (-9 + 2 l + 2 n)}$    & $2$ & $\frac{3}{10}$  \\
        $(l \geq 4,l)$ & $\frac{-2810 + 887 l + 512 n}{48 (-515 + 128 l + 128 n)}$ & $\frac{-9842 + 2071 l + 2800 n}{200 (-113 + 28 l + 28 n)}$    & $2$ & $\frac{3}{10}$  \\
        $(l \geq 3,l+1)$ & $\frac{-27 + 10 l + 4 n}{12 (-15 + 4 l + 4 n)}$ & $\frac{-29 + 5 l + 13 n}{2 (-45 + 13 l + 13 n)}$    & $2$ & $\frac{3}{10}$  \\
		\botrule
	\end{tabular}
\end{table}
\renewcommand{\arraystretch}{1.0}

\begin{lemma} \label{lem::bounds}
Let $\Re \lambda \geq \modebound$ and let the pair $(l,m)$ be such that $l > 0$. If $n_0, \bar{a}_n, \bar{b}_n$ and $u$ are chosen according to Table~\ref{table_bounds} depending on the value of $(l,m)$, then
\begin{gather}
	\left| a_n \right| \leq \bar{a}_n, \\
	\left| b_n \right| \leq \bar{b}_n,
\end{gather}
for $n \geq n_0$ and in addition,
\begin{equation}
	\left| e_{n_0} \right| \leq u.
\end{equation}
\end{lemma}
\begin{proof}
Let us choose $n_0$ as in Table~\ref{table_bounds}. We first prove that the coefficients $a_n, b_n$ as well as $e_{n_0}$ are analytic functions of the growth rate $\lambda$ when $\Re \lambda \geq 0$.
	
	\begin{claim*}
		$\tilde{r}_n^{-1}$ is analytic as a function of $\lambda$ in the domain $\lambda \in \{ z \in \C : \Re z \geq 0 \}$ for all $n \geq 1$.
	\end{claim*}
	\noindent
	This follows simply by observing that $\tilde{r}_n$ is a quadratic polynomial in $\lambda$ for fixed $n,l,m$, computing the roots of this polynomial and checking that they lie in $\{ z \in \C : \Re z < 0 \}$ for each $n$ and for each of the different parameters $l$ and $m$. Indeed, both roots are negative real numbers in all cases. Since the explicit expressions are a bit cumbersome, we carry the calculation out in Appendix~\ref{sec::zeroes}.
	
	\begin{claim*}
		$r_{n_0}$ is analytic as a function of $\lambda$ in the domain $\lambda \in \{ z \in \C : \Re z \geq 0 \}$.
	\end{claim*}
	\noindent
	By explicitly computing $r_{n_0}$ we observe that it is a rational function of $\lambda$. Therefore we need to show that its denominator has its roots contained in $\{ \lambda \in \C : \Re \lambda < 0 \}$. This is easily done by using Wall's criterion, see Appendix~\ref{sec::wall}.

    \vspace{1em}
    \noindent
    From the explicit form~(\ref{equ::recursioncoeffA}),~(\ref{equ::recursioncoeffB}) of $A_n$ and $B_n$ provided above it is clear that $\lambda$ only occurs in the numerator and it does so polynomially. In particular $A_n$ and $B_n$ are analytic as functions of $\lambda \in \C$. It therefore follows immediately that $a_n$ and $b_n$ are analytic as functions of $\lambda \in \{ z \in \C : \Re z \geq 0 \}$. Similarly, the analyticity of $r_{n_0}$ and $\tilde{r}_n^{-1}$ implies that of $e_{n_0}$. According to the Phragmen-Lindel\"{o}f principle, it therefore suffices to show the required bounds for $\lambda$ on the imaginary line $\{z \in \C : \Re z = 0 \}$ in order to obtain them for the domain $\{ z \in \C : \Re z \geq 0 \}$.
	
	Now suppose $\lambda=  i t$, $t \in \R$. We demonstrate the argument on the example of $a_n$ for the case $(l,m)=(l,l-1)$, $l \geq 4$, but the argument is identical for $b_n$, for $e_{n_0}$ and all other cases. We explicitly compute $|a_n|^2(i t)$ and note that this is a rational function of $t^2$ whose coefficients are polynomials in $n$ and $l$ with only integer coefficients, say $|a_n|^2(i t) = \frac{F(t^2)}{G(t^2)}$. Observe that all of our bounds are also rational functions of $n$ and $l$ with integer coefficients. Therefore to show $|a_n| \, \leq \frac{x}{y}$, say, we can equivalently show $F y^2 - G x^2 \leq 0$. This expression is again a polynomial in $t^2$ with coefficients that are polynomials in $n$ and $l$. When we shift the variables $n \mapsto n + 2$ and $l \mapsto l+4$, we then find that all of the non-zero coefficients appearing here are negative integers, therefore demonstrating the negativity of $F y^2 - G x^2$ for $n \geq 2$ and $l \geq 4$. Since the explicit expressions are cumbersome, we provide the explicit form of the polynomial $F y^2 - G x^2$ for all cases of $(l,m)$ and for $a_n$, $b_n$ and $e_{n_0}$ in a digital format, see Appendix~\ref{sec::github}.
\end{proof}

\begin{lemma} \label{lem:relerrorbound}
Let $\Re \lambda \geq \modebound$ and let the pair $(l,m)$ be such that $l > 0$. Then 
\begin{equation}
    \left| e_n \right| \leq \frac{1}{3},
\end{equation}
for all $n \geq n_0$, where $n_0$ is chosen as in Table~\ref{table_bounds}.
\end{lemma}
\begin{proof}
Note that for $x \in \C, y \in \R$ with $| x | \, \leq y < 1$
\begin{equation}
\left\lvert \frac{x}{1+x} \right\rvert \leq \frac{|x|}{1-|x|} \leq \frac{y}{1-y}.
\end{equation}
Let us assume that we have made a choice of $0< y < 1$ and that we have obtained bounds of the form $|a_n| \, \leq \bar{a}_n$ and $|b_n| \, \leq \bar{b}_n$. Then in order to close the argument we see that the crucial property is that $\bar{a}_n, \bar{b}_n$ and $y$ satisfy
\begin{equation}
\bar{a}_n + \bar{b}_n \frac{y}{1-y} \leq y,
\end{equation}
or equivalently
\begin{equation}
    y^2 + (\bar{b}_n- \bar{a}_n - 1) y + \bar{a}_n <0.
\end{equation}
The reader may convince themselves that for all values of $(l,m)$, the value $y=\frac{1}{3}$ satisfies this condition for all $n \geq n_0$ and our choice of bounds $\bar{a}_n$ and $\bar{b}_n$.
\end{proof}

The results of this section are readily combined to prove Proposition~\ref{prop:heunsols}.

\begin{proof}[Proof of Proposition~\ref{prop:heunsols}]
We begin by noting that Lemma~\ref{lem:poincare} together with Lemma~\ref{lem:quasisoln_limit} implies that the limit $\lim_{n \rightarrow \infty} e_n$ exists. By Lemma~\ref{lem:relerrorbound}, $\lim_{n \rightarrow \infty} \left| e_n \right| \leq \frac{1}{3}$. Finally we note that Lemma~\ref{lem:poincare} implies $\lim_{n \rightarrow \infty} r_n \in \{1, \frac{1}{2} \}$, and a quick computation shows that if $\lim_{n \rightarrow \infty} r_n = \frac{1}{2}$, then necessarily $\lim_{n \rightarrow \infty} \left| e_n \right| = \frac{1}{2}$, which is impossible. Therefore, $\lim_{n \rightarrow \infty} r_n = 1$, as claimed.
\end{proof}

\subsection{Proof of Theorem~\ref{maintheorem}}
Finally, we note that Proposition~\ref{prop:heunsols} together with Proposition~\ref{prop:simplecase}, combined with the arguments in Section~\ref{sec:heunform} and Lemma~\ref{lem:susytransform}, establishes the proof of Theorem~\ref{maintheorem}.

\begin{appendices}

\section{Explicit form of generators} \label{appendix::matrices}
Here we provide the explicit form of the rotation and Lorentz boost matrices used in the computations in Chapter~\ref{sec:symmetries_and_modes}. We fix a set of generators of the Lie algebra $\text{so}(3)$ here:
\begin{gather}
F_1 = \left( \begin{matrix}
0 & 0& 0 \\
0 & 0& -1 \\
0 & 1 & 0
\end{matrix} \right), \quad F_2 = \left( \begin{matrix}
0 & 0& 1 \\
0 & 0& 0 \\
-1 & 0 & 0
\end{matrix} \right), \quad F_3 = \left( \begin{matrix}
0 & -1 & 0 \\
1 & 0& 0 \\
0 & 0 & 0
\end{matrix} \right).
\end{gather}
The generators of the Lie algebra $\text{so}(4)$ we choose are the following
\begin{equation}
\begin{gathered}
\mathbf{F}_1 = \left(
\begin{array}{cccc}
 0 & 0 & 0 & 0 \\
 0 & 0 & -1 & 0 \\
 0 & 1 & 0 & 0 \\
 0 & 0 & 0 & 0 \\
\end{array}
\right), \quad \mathbf{F}_2 =\left(
\begin{array}{cccc}
 0 & 0 & -1 & 0 \\
 0 & 0 & 0 & 0 \\
 1 & 0 & 0 & 0 \\
 0 & 0 & 0 & 0 \\
\end{array}
\right), \\
\mathbf{F}_3 = \left(
\begin{array}{cccc}
 0 & -1 & 0 & 0 \\
 1 & 0 & 0 & 0 \\
 0 & 0 & 0 & 0 \\
 0 & 0 & 0 & 0 \\
\end{array}
\right), \quad \mathbf{F}_4 = \left(
\begin{array}{cccc}
 0 & 0 & 0 & -1 \\
 0 & 0 & 0 & 0 \\
 0 & 0 & 0 & 0 \\
 1 & 0 & 0 & 0 \\
\end{array}
\right), \\
\mathbf{F}_5 = \left(
\begin{array}{cccc}
 0 & 0 & 0 & 0 \\
 0 & 0 & 0 & -1 \\
 0 & 0 & 0 & 0 \\
 0 & 1 & 0 & 0 \\
\end{array}
\right), \quad \mathbf{F}_6 = \left(
\begin{array}{cccc}
 0 & 0 & 0 & 0 \\
 0 & 0 & 0 & 0 \\
 0 & 0 & 0 & -1 \\
 0 & 0 & 1 & 0 \\
\end{array}
\right).
\end{gathered}
\end{equation}
The Lorentz boosts are given by
\begin{equation}
\begin{gathered}
\Lambda_1(\alpha) = \left( \begin{matrix} \cosh \alpha & -\sinh \alpha & 0 & 0 \\[3pt] -\sinh \alpha & \cosh \alpha & 0 & 0 \\[3pt] 0 & 0 & 1 & 0 \\[3pt] 0 & 0 & 0 & 1\end{matrix} \right), \quad \Lambda_2(\alpha) = \left( \begin{matrix} \cosh \alpha & 0 & \sinh \alpha & 0 \\[3pt] 0 & 1 & 0 & 0 \\[3pt] \sinh \alpha & 0 & \cosh \alpha & 0 \\[3pt] 0 & 0 & 0 & 1 \end{matrix} \right), \\[0.1 cm]
\Lambda_3(\alpha) = \left( \begin{matrix} \cosh \alpha & 0 & 0 & -\sinh \alpha \\[3pt] 0 & 1 & 0 & 0 \\[3pt] 0 & 0 & 1 & 0 \\[3pt] -\sinh \alpha & 0 & 0 & \cosh \alpha \end{matrix} \right).
\end{gathered}
\end{equation}

\section{Omitted computations}
\subsection{Roots of the quasi-solutions \texorpdfstring{$\tilde{r}_n$}{rn}} \label{sec::zeroes}
In this section we provide the explicit formulas for the roots of the approximations $\tilde{r}_n$ as required in the first step of the proof of Lemma~\ref{lem::bounds}. For fixed $n,l$ and $m$, $\tilde{r}_n$ is a quadratic polynomial in $\lambda$. Therefore its roots are readily computed. We provide the explicit form of the roots in this section for all cases. \\

\noindent \textbf{Case $l \geq 3, m= l-1$:}
Here the roots are directly computed to be located at
\begin{equation}
\begin{split}
\lambda=  &\frac{l (8 n+8)+8 n^2+20 n+12}{4 l n+4 l+4 n^2+10 n+6} \Bigg( -(l+2 n+1)\\
&\pm \sqrt{\frac{6 l^2 n+20 l^2+30 l n^2+199 l n+162 l+48 n^3+262 n^2+313 n+218}{8 (3 n+10)}} \Bigg).
\end{split}
\end{equation}
Both of these values are in fact negative real numbers. To see this first observe that the expression in the square root is clearly positive for $n \geq 1$. Therefore we only need to check that the last expression in the first line is greater in absolute value than the square root. We therefore calculate the difference of the square of the last expression from the first line and the expression inside the square root. We obtain
\begin{equation}
\frac{18 l^2 n+60 l^2+66 l n^2+169 l n-2 l+48 n^3+154 n^2+31 n-138}{8 (3 n+10)}
\end{equation}
which is easily verified to be positive for all $l \geq 3$ and $n \geq 2$ by shifting the variables and expanding. For the remaining cases we will simply provide the formula and omit the proof of negativity, as it is a similar computation. \\

\noindent \textbf{Case $l \geq 3, m= l$:}
\begin{equation}
\begin{split}
\lambda=  &\frac{l (8 n+8)+8 n^2+20 n+12}{4 l n+4 l+4 n^2+10 n+6} \Bigg( -(l+2 n+1)\\
&\pm \sqrt{\frac{2 l^2 n+12 l^2+10 l n^2+95 l n+50 l+16 n^3+132 n^2+106 n+60}{8 (n+6)}} \Bigg).
\end{split}
\end{equation} 

\noindent \textbf{Case $l \geq 3, m= l+1$:}
\begin{equation}
\begin{split}
\lambda=  &\frac{l (8 n+8)+8 n^2+20 n+12}{4 l n+4 l+4 n^2+10 n+6} \Bigg( -(l+2 n+1)\\
&\pm \sqrt{\frac{4 l^2 n+30 l^2+20 l n^2+208 l n+19 l+32 n^3+300 n^2+116 n-9}{8 (2 n+15)}} \Bigg).
\end{split}
\end{equation}

\noindent \textbf{Case $l =1, m= 1$:}
\begin{equation}
\lambda= \frac{8 n^2+36 n+28}{14+18n+4n^2} \left( -(3+2n) \pm \sqrt{\frac{6 n^3+35 n^2+75 n+51}{3n+8}} \right).
\end{equation}

\noindent \textbf{Case $l =1, m= 2$:}
\begin{equation}
\lambda= \frac{8 n^2+28 n+20}{10+14n+4n^2} \left( -(2+2n) \pm \sqrt{\frac{-2 + 13 n + 17 n^2 + 2 n^3}{n+7}} \right).
\end{equation}

\noindent \textbf{Case $l =2, m= 1$:}
\begin{equation}
\lambda= \frac{8 n^2+44 n+36}{18+22n+4n^2} \left( -(4+2n) \pm \sqrt{\frac{4 n^3+40 n^2+107 n+111}{2n+12}} \right).
\end{equation}

\noindent \textbf{Case $l =2, m= 2$:}
\begin{equation}
\lambda= \frac{8 n^2+36 n+28}{4 n^2+18 n+14} \left( -(3+2n) \pm \sqrt{\frac{12 n^3+98 n^2+162 n+105}{6 n+35}} \right).
\end{equation}

\noindent \textbf{Case $l =2, m= 3$:}
\begin{equation}
\lambda= \frac{8 n^2+36 n+28}{4 n^2+18 n+14} \left( -(3+2n) \pm \sqrt{\frac{8 n^3+116 n^2+194 n+129}{4 n+47}} \right).
\end{equation}

\subsection{Analyticity of \texorpdfstring{$r_{n_0}$}{rn0}} \label{sec::wall}
In this section we provide the proof of analyticity of $r_{n_0}$ as a function of $\lambda$ in the region $\Re \lambda \geq \modebound$, as required in the second part of the proof of Lemma~\ref{lem::bounds}. Recall that the initial value $n_0$ we choose depends on $l$ and $m$, as specified in Table~\ref{table_bounds}. The proof of analyticity is based on Wall's criterion~\cite{wall}.

First we note that $r_n$ is a rational function of $\lambda$, so let $d_n$ be its denominator, a polynomial in $\lambda$ of degree $2n$. To show analyticity of $r_n$ as a function $\lambda$ when $\Re \lambda \geq 0$, we need to show that $d_n$ only vanishes in the region $\Re \lambda < 0$. Let $\hat{d}_n$ denote the polynomial obtained from $d_n$ by setting all coefficients of even powers of $\lambda$ to be zero. By successive polynomial division (using $\lambda$ as the variable), we can obtain a continued fraction expansion of the quotient $\frac{\hat{d}_n}{d_n}$. Wall's criterion implies that this expansion takes the form:
\begin{equation}
\frac{\hat{d}_n}{d_n} = \frac{1}{1+x_1  \lambda +} \, \frac{1}{x_2  \lambda +} \, \frac{1}{x_3  \lambda +} \dots \frac{1}{x_{2n}  \lambda},
\end{equation}
where all coefficients satisfy $x_i > 0$ if and only if $d_n$ has all its roots in the region $\Re \lambda < 0$. Therefore, we compute this continued fractions expansion in each case and list the resulting coefficients $x_i$. \\

\noindent \textbf{Case $l =1, m=1$:} In this case $r_2$ is given by
\begin{equation}
r_2(\lambda)= \frac{\lambda ^6+60 \lambda ^5+1201 \lambda ^4+10152 \lambda ^3+37851 \lambda ^2+55580 \lambda +19635}{132 \left(\lambda ^4+32 \lambda ^3+266 \lambda ^2+592 \lambda +245\right)},
\end{equation}
so that $d_2$ is a polynomial of degree $4$. Here we find
\begin{equation}
x_1= \frac{1}{32}, \quad x_2 = \frac{64}{495}, \quad x_3=\frac{49005 }{110944}, \quad x_4= \frac{55472}{24255}.
\end{equation}

\noindent \textbf{Case $l =1, m=2$:}
Here $r_4$ has the numerator given by
\begin{equation}
\begin{gathered}
\lambda ^{10}+120 \lambda ^9+5655 \lambda ^8+138560 \lambda ^7+1969418 \lambda ^6+17090160 \lambda ^5 \\
+92390286 \lambda ^4+310928256 \lambda ^3+641783397 \lambda ^2+787540056 \lambda +488363755,
\end{gathered}
\end{equation}
and denominator
\begin{equation}
\begin{gathered}
d_4= 260 (\lambda ^8+80 \lambda ^7+2356 \lambda ^6+33584 \lambda ^5+256238 \lambda ^4 \\
+1088432 \lambda ^3+2600580 \lambda ^2+3504848 \lambda +2391129).
\end{gathered}
\end{equation}
For the continued fractions expansion we find
\begin{equation}
\begin{gathered}
x_1= \frac{1}{80}, \quad x_2 = \frac{400}{9681}, \quad x_3=\frac{13388823}{162909760}, \quad x_4= \frac{16587243689536}{113962657805643} \\[0.2cm]
x_5=\frac{47531204298703335341887}{191285306692074662805504},\\[0.2cm]
x_6 = \frac{17233719835941124753004235620352}{40133868683257984044230012780567}, \\[0.2cm]
x_7 = \frac{4841112694132470445768992098446182544921}{6482346130187177237572701069669289181184} \\[0.2cm]
x_8 = \frac{141331078934075674653925376  }{166370224608919186154542269}.
\end{gathered}
\end{equation}

\noindent \textbf{Case $l =2, m=1$:} In this case $r_2$ is given by
\begin{equation}
r_2(\lambda)=\frac{\lambda ^6+72 \lambda ^5+1813 \lambda ^4+20400 \lambda ^3+108019 \lambda ^2+251784 \lambda +194103}{156 \left(\lambda ^4+40 \lambda ^3+458 \lambda ^2+1688 \lambda +1701\right)},
\end{equation}
so that $d_2$ is a polynomial of degree $4$. Here we find
\begin{equation}
x_1= \frac{1}{40}, \quad x_2 = \frac{200}{2079}, \quad x_3=\frac{22869 }{83840}, \quad x_4= \frac{16768}{18711}.
\end{equation}

\noindent \textbf{Case $l =2, m=2$:} In this case $r_3$ has the numerator
\begin{equation}
\begin{gathered}
\lambda ^8+96 \lambda ^7+3456 \lambda ^6+60912 \lambda ^5+574976 \lambda ^4 \\
+2974208 \lambda ^3+8253120 \lambda ^2+11432704 \lambda +6432768,
\end{gathered}
\end{equation}
and denominator
\begin{equation}
\begin{gathered}
d_3 = 208 \left(\lambda ^6+60 \lambda ^5+1216 \lambda ^4
+10488 \lambda ^3+39936 \lambda ^2+65984 \lambda +40704\right).
\end{gathered}
\end{equation}
Here we find the following coefficients for the expansion:
\begin{equation}
\begin{gathered}
x_1= \frac{1}{60}, \quad x_2 = \frac{150}{2603}, \quad x_3=\frac{6775609 }{53687060}, \quad x_4= \frac{21617253085827}{80716560627608} \\[0.2cm]
x_5 = \frac{30048789523708855778}{51443003963743308357}, \quad x_6 = \frac{6388007832023}{4930439162424}.
\end{gathered}
\end{equation}

\noindent \textbf{Case $l =3, m=3$:} Here we find that $r_3$ has numerator
\begin{equation}
\begin{gathered}
\lambda ^8+112 \lambda ^7+4804 \lambda ^6+103408 \lambda ^5 +1229214 \lambda ^4 \\
+8329808 \lambda ^3+31803380 \lambda ^2+63649104 \lambda +52369065,
\end{gathered}
\end{equation}
and denominator
\begin{equation}
\begin{gathered}
d_3=240 \left(\lambda ^6+72 \lambda ^5+1813 \lambda ^4+20400 \lambda ^3+109011 \lambda ^2+272264 \lambda +260055\right).
\end{gathered}
\end{equation}
Here we obtain
\begin{equation}
\begin{gathered}
x_1= \frac{1}{72}, \quad x_2 = \frac{216 }{4589}, \quad x_3=\frac{21058921  }{212658048}, \quad x_4= \frac{117769389008256 }{605966643139039} \\[0.2cm]
x_5 = \frac{17436580563514884792601 }{45956710403723331293184}, \quad x_6 = \frac{27661586227277824 }{34339650333737805}.
\end{gathered}
\end{equation}

\noindent \textbf{Case $l\geq 3, m=l-1$:}
In this case $r_2$ has numerator
\begin{equation}
\begin{gathered}
\lambda ^5+38 \lambda ^4+440 \lambda ^3+1816 \lambda ^2+2096 \lambda +27 l^5+81 \lambda  l^4+306 l^4 \\
+90 \lambda ^2 l^3+804 \lambda  l^3+1224 l^3
+46 \lambda ^3 l^2+744 \lambda ^2 l^2+2792 \lambda  l^2 \\ 
+2104 l^2+11 \lambda ^4 l +284 \lambda ^3 l+2008 \lambda ^2 l+3984 \lambda  l+1264 l,
\end{gathered}
\end{equation}
and denominator
\begin{equation}
\begin{gathered}
d_2=12 (2 l+7) \left(\lambda ^3+18 \lambda ^2+68 \lambda +9 l^3+15 \lambda  l^2+46 l^2+7 \lambda ^2 l+64 \lambda  l+52 l\right),
\end{gathered}
\end{equation}
so that $d_2$ is a polynomial of degree $3$. Here we find
\begin{equation}
\begin{gathered}
x_1= \frac{1}{7 l+18}, \quad x_2 = \frac{(7 l+18)^2}{8 \left(12 l^3+84 l^2+197 l+153\right)}, \\[0.2cm]
x_3=\frac{96 l^3+672 l^2+1576 l+1224}{(7 l+18) \left(9 l^3+46 l^2+52 l\right)}.
\end{gathered}
\end{equation}

\noindent \textbf{Case $l\geq 4, m=l$:} Here $r_2$ has numerator
\begin{equation}
\begin{gathered}
\lambda ^6+36 \lambda ^5+364 \lambda ^4+936 \lambda ^3-1536 \lambda ^2-4192 \lambda +27 l^6+ \\
108 \lambda  l^5+360 l^5+171 \lambda ^2 l^4 +1236 \lambda  l^4+1524 l^4+136 \lambda ^3 l^3+1632 \lambda ^2 l^3 \\
+4424 \lambda  l^3+1944 l^3+57 \lambda ^4 l^2+1032 \lambda ^3 l^2 +4704 \lambda ^2 l^2+4440 \lambda  l^2-1440 l^2 \\ 
+12 \lambda ^5 l +312 \lambda ^4 l+2168 \lambda ^3 l+3432 \lambda ^2 l-3104 \lambda  l-3360 l,
\end{gathered}
\end{equation}
and denominator given by
\begin{equation}
\begin{gathered}
d_2=12 (2 l+7) (\lambda ^4+16 \lambda ^3+32 \lambda ^2-136 \lambda +9 l^4+24 \lambda  l^3+52 l^3+22 \lambda ^2 l^2 \\
+112 \lambda  l^2+24 l^2+8 \lambda ^3 l+76 \lambda ^2 l+56 \lambda  l-120 l).
\end{gathered}
\end{equation}
In this case, $d_2$ is again a polynomial of degree four in $\lambda$ and we find
\begin{equation}
\begin{gathered}
x_1= \frac{1}{8 l+16}, \quad x_2 = \frac{8 (l+2)^2}{19 l^3+106 l^2+177 l+81}, \\[0.2cm]
x_3=\frac{\left(19 l^3+106 l^2+177 l+81\right)^2}{8 (l+2) \left(48 l^6+496 l^5+1880 l^4+2956 l^3+955 l^2-1962 l-1377\right)}, \\[0.2cm]
x_4 = \frac{8 \left(48 l^6+496 l^5+1880 l^4+2956 l^3+955 l^2-1962 l-1377\right)}{l \left(9 l^3+52 l^2+24 l-120\right) \left(19 l^3+106 l^2+177 l+81\right)}.
\end{gathered}
\end{equation}
Note that even though it is not immediately apparent, one can easily see that the coefficients $x_3,x_4$ are positive for $l \geq 2$ by inserting $l+2$ for $l$ and expanding. Then all coefficients become positive. \\

\noindent \textbf{Case $l\geq 2, m=l+1$:} Here $r_2$ has numerator
\begin{equation}
\begin{gathered}
\lambda ^6+36 \lambda ^5+376 \lambda ^4+1224 \lambda ^3+160 \lambda ^2-2112 \lambda \\
+27 l^6+108 \lambda  l^5+468 l^5 +171 \lambda ^2 l^4 +1524 \lambda  l^4+2832 l^4+136 \lambda ^3 l^3+1896 \lambda ^2 l^3 \\
+7112 \lambda  l^3+7112 l^3+57 \lambda ^4 l^2+1128 \lambda ^3 l^2 +6456 \lambda ^2 l^2+12120 \lambda  l^2 \\
+6464 l^2+12 \lambda ^5 l+324 \lambda ^4 l+2552 \lambda ^3 l+6616 \lambda ^2 l+4256 \lambda  l+576 l,
\end{gathered}
\end{equation}
and denominator
\begin{equation}
\begin{gathered}
d_2=12 (2 l+7) (\lambda ^4+16 \lambda ^3+40 \lambda ^2-72 \lambda +9 l^4+24 \lambda  l^3+76 l^3 \\
+22 \lambda ^2 l^2+144 \lambda  l^2+144 l^2+8 \lambda ^3 l+84 \lambda ^2 l+152 \lambda  l-24 l).
\end{gathered}
\end{equation}
For the continued fraction expansion of $d_2$ we find
\begin{equation}
\begin{gathered}
x_1= \frac{1}{8 l+16}, \quad x_2 = \frac{8 (l+2)^2}{19 l^3+110 l^2+189 l+89}, \\[0.2cm]
x_3=\frac{\left(19 l^3+110 l^2+189 l+89\right)^2}{8 (l+2) \left(48 l^6+560 l^5+2424 l^4+4732 l^3+3723 l^2+86 l-801\right)}, \\[0.2cm]
x_4 = \frac{8 \left(48 l^6+560 l^5+2424 l^4+4732 l^3+3723 l^2+86 l-801\right)}{l \left(9 l^3+76 l^2+144 l-24\right) \left(19 l^3+110 l^2+189 l+89\right)}.
\end{gathered}
\end{equation}
Here it is again not immediately apparent that $x_3$ and $x_4$ are positive. This can be remedied by shifting $l$ to $l+1$. Then all coefficients become positive.

\subsection{Supplementary online material} \label{sec::github}
In this section, we describe the contents of the files made available as supplementary online material. At various points in the main body of this work, the authors have decided to provide explicit expressions in a digital format rather than in the text. The expressions involved are cumbersome, and making them available in a digital format will facilitate reproducibility. We provide the following list of supplementary files:
\begin{center}
 \texttt{11.csv}, \texttt{12.csv}, \texttt{21.csv}, \texttt{22.csv}, \texttt{23.csv}, \\
 \texttt{32.csv}, \texttt{33.csv}, \texttt{l1.csv}, \texttt{l2.csv}, \texttt{l3.csv}.
\end{center}
The files \texttt{l1.csv}, \texttt{l2.csv} and \texttt{l3.csv} correspond to the cases $(l,m)=(l,l-1)$ with $l \geq 4$, $(l,m)=(l,l)$ with $l \geq 4$, respectively $(l,m)=(l,l+1)$ with $l \geq 3$. The remaining files named in the form \texttt{lm.csv} correspond to the case $(l,m)$ provided in the name. A description of what is contained in each file is provided in Table~\ref{table:filecontents}. The contents of each file are machine readable (for instance with the Mathematica software package).

\renewcommand{\arraystretch}{1.6}
\begin{table}[ht]
\caption{Contents of the supplementary files. We represent the growth rate $\lambda$ as the variable \texttt{x} throughout the files.}\label{table:filecontents}%
\begin{tabular}{@{}lp{10cm}@{}}
\toprule
Variable name & Description \\
\midrule
\texttt{A} & Coefficient $A_n$ as defined in~\eqref{equ::recursioncoeffA}. \\
\texttt{B } & Coefficient $B_n$ as defined in~\eqref{equ::recursioncoeffB}. \\
\texttt{n0} & Value of $n_0$, as in Table~\ref{table_bounds}. \\
$\texttt{r}_{n_0}$ & Explicit form of $r_{n_0}$. \\
\texttt{rtilde} & Explicit form of the quasi-solution $\tilde{r}_n$ defined in Table~\ref{table_approximations}. \\
\texttt{a} & Explicit form of the coefficient $a_n$ defined in~\eqref{equ:anbn} \\
\texttt{b} & Explicit form of the coefficient $b_n$ defined in~\eqref{equ:anbn} \\
\texttt{esta} & The polynomial $F y^2 - G x^2$ from the proof of Lemma~\ref{lem::bounds} for the coefficient $a_n$. We shift $n \mapsto n + n_0$. In the files \texttt{l1.csv}, \texttt{l2.csv} respectively \texttt{l3.csv} where $l$ appears as a variable, we also shift $l \mapsto l + l_0$ where $l_0 = 4, 4,3$ respectively. \\
\texttt{estb} & Analogous to \texttt{esta} but for the coefficient $b_n$. The same shifts in $n$ and $l$ are applied as for \texttt{esta} above. \\
\texttt{esterror} & Analogous to \texttt{esta} but for $|e_{n_0}|$. Since $n = n_0$ here, $n$ does not appear as a variable. In the file \texttt{l1.csv} we do not shift the value of $l$, in \texttt{l2.csv} we shift $l \mapsto l+4$ and in \texttt{l3.csv} we shift $l \mapsto l+2$. \\
\botrule
\end{tabular}
\end{table}
\renewcommand{\arraystretch}{1.0}

\begin{comment}
\noindent Thus, by consulting the provided files, the reader may verify the validity of the claimed bounds on the coefficients and the initial value of the error. As remarked above, in every computation except one, the negativity of the polynomial is immediately clear from the fact that every non-zero coefficient is a negative integer. In the case $(l,m)=(1,2)$ the error estimate gives us the polynomial
\begin{align*}
\texttt{esterror} = &-32225063143731938775+1229703226782849704 t^2 \\
&-1504371751505412669 t^4 -384997689421455888 t^6 \\
&-27845993942231878 t^8-896965123990016 t^{10} \\
&-13901512438618 t^{12}-99535465456 t^{14} \\
&-299917523 t^{16}-340648 t^{18}-121 t^{20}
\end{align*}
so that the negativity is not immediately apparent. However, it can be easily seen that when $|t| \, \leq 5$, the sum of the first two terms is negative. Thus we may assume that $|t| \, > 5$. Define $x = t^2$ and write the polynomial as a function of $x$. Then shifting the variable $x \mapsto x + 25$ and expanding the polynomial one arrives at
\begin{gather*}
-121 x^{10}-370898 x^9-379966448 x^8-167410425056 x^7-37025801827168 x^6 \\
-4568086848705216 x^5-333353001325840128 x^4-14651780151181618688 x^3 \\
-378945095134180006144 x^2-5287882315793456042496 x-30642585141535849676800
\end{gather*}
which is evidently negative when $x \geq 0$, i.e. when $|t| \, \geq 5$.
\end{comment}

\end{appendices}

\clearpage
\bibliography{sn-bibliography}


\begin{thebibliography}{79}
\ifx \bisbn   \undefined \def \bisbn  #1{ISBN #1}\fi
\ifx \binits  \undefined \def \binits#1{#1}\fi
\ifx \bauthor  \undefined \def \bauthor#1{#1}\fi
\ifx \batitle  \undefined \def \batitle#1{#1}\fi
\ifx \bjtitle  \undefined \def \bjtitle#1{#1}\fi
\ifx \bvolume  \undefined \def \bvolume#1{\textbf{#1}}\fi
\ifx \byear  \undefined \def \byear#1{#1}\fi
\ifx \bissue  \undefined \def \bissue#1{#1}\fi
\ifx \bfpage  \undefined \def \bfpage#1{#1}\fi
\ifx \blpage  \undefined \def \blpage #1{#1}\fi
\ifx \burl  \undefined \def \burl#1{\textsf{#1}}\fi
\ifx \doiurl  \undefined \def \doiurl#1{\url{https://doi.org/#1}}\fi
\ifx \betal  \undefined \def \betal{\textit{et al.}}\fi
\ifx \binstitute  \undefined \def \binstitute#1{#1}\fi
\ifx \binstitutionaled  \undefined \def \binstitutionaled#1{#1}\fi
\ifx \bctitle  \undefined \def \bctitle#1{#1}\fi
\ifx \beditor  \undefined \def \beditor#1{#1}\fi
\ifx \bpublisher  \undefined \def \bpublisher#1{#1}\fi
\ifx \bbtitle  \undefined \def \bbtitle#1{#1}\fi
\ifx \bedition  \undefined \def \bedition#1{#1}\fi
\ifx \bseriesno  \undefined \def \bseriesno#1{#1}\fi
\ifx \blocation  \undefined \def \blocation#1{#1}\fi
\ifx \bsertitle  \undefined \def \bsertitle#1{#1}\fi
\ifx \bsnm \undefined \def \bsnm#1{#1}\fi
\ifx \bsuffix \undefined \def \bsuffix#1{#1}\fi
\ifx \bparticle \undefined \def \bparticle#1{#1}\fi
\ifx \barticle \undefined \def \barticle#1{#1}\fi
\bibcommenthead
\ifx \bconfdate \undefined \def \bconfdate #1{#1}\fi
\ifx \botherref \undefined \def \botherref #1{#1}\fi
\ifx \url \undefined \def \url#1{\textsf{#1}}\fi
\ifx \bchapter \undefined \def \bchapter#1{#1}\fi
\ifx \bbook \undefined \def \bbook#1{#1}\fi
\ifx \bcomment \undefined \def \bcomment#1{#1}\fi
\ifx \oauthor \undefined \def \oauthor#1{#1}\fi
\ifx \citeauthoryear \undefined \def \citeauthoryear#1{#1}\fi
\ifx \endbibitem  \undefined \def \endbibitem {}\fi
\ifx \bconflocation  \undefined \def \bconflocation#1{#1}\fi
\ifx \arxivurl  \undefined \def \arxivurl#1{\textsf{#1}}\fi
\csname PreBibitemsHook\endcsname

%%% 1
\bibitem[\protect\citeauthoryear{Gell-Mann and L{\'e}vy}{1960}]{gell1960axial}
\begin{barticle}
\bauthor{\bsnm{Gell-Mann}, \binits{M.}},
\bauthor{\bsnm{L{\'e}vy}, \binits{M.}}:
\batitle{The axial vector current in beta decay}.
\bjtitle{Il Nuovo Cimento (1955-1965)}
\bvolume{16}(\bissue{4}),
\bfpage{705}--\blpage{726}
(\byear{1960})
\end{barticle}
\endbibitem

%%% 2
\bibitem[\protect\citeauthoryear{Turok and Spergel}{1990}]{turok}
\begin{barticle}
\bauthor{\bsnm{Turok}, \binits{N.}},
\bauthor{\bsnm{Spergel}, \binits{D.}}:
\batitle{{Global texture and the microwave background}}.
\bjtitle{Physical Review Letters}
\bvolume{64}(\bissue{23}),
\bfpage{2736}
(\byear{1990})
\end{barticle}
\endbibitem

%%% 3
\bibitem[\protect\citeauthoryear{Shatah}{1988}]{shatah}
\begin{barticle}
\bauthor{\bsnm{Shatah}, \binits{J.}}:
\batitle{{Weak solutions and development of singularities of the {${\rm SU}(2)$} {$\sigma$}-model}}.
\bjtitle{Comm. Pure Appl. Math.}
\bvolume{41}(\bissue{4}),
\bfpage{459}--\blpage{469}
(\byear{1988})
\end{barticle}
\endbibitem

%%% 4
\bibitem[\protect\citeauthoryear{Bizo\'n et~al.}{2000}]{bizon}
\begin{barticle}
\bauthor{\bsnm{Bizo\'n}, \binits{P.}},
\bauthor{\bsnm{Chmaj}, \binits{T.}},
\bauthor{\bsnm{Tabor}, \binits{Z.a.}}:
\batitle{{Dispersion and collapse of wave maps}}.
\bjtitle{Nonlinearity}
\bvolume{13}(\bissue{4}),
\bfpage{1411}--\blpage{1423}
(\byear{2000})
\end{barticle}
\endbibitem

%%% 5
\bibitem[\protect\citeauthoryear{Costin et~al.}{2017}]{glogic2}
\begin{barticle}
\bauthor{\bsnm{Costin}, \binits{O.}},
\bauthor{\bsnm{Donninger}, \binits{R.}},
\bauthor{\bsnm{Glogi\'c}, \binits{I.}}:
\batitle{{Mode stability of self-similar wave maps in higher dimensions}}.
\bjtitle{Comm. Math. Phys.}
\bvolume{351}(\bissue{3}),
\bfpage{959}--\blpage{972}
(\byear{2017})
\end{barticle}
\endbibitem

%%% 6
\bibitem[\protect\citeauthoryear{Glogi\'c}{2018}]{glogicthesis}
\begin{botherref}
\oauthor{\bsnm{Glogi\'c}, \binits{I.}}:
{On the {E}xistence and {S}tability of {S}elf-{S}imilar {B}lowup in {N}onlinear {W}ave {E}quations}.
PhD thesis,
The Ohio State University
(2018)
\end{botherref}
\endbibitem

%%% 7
\bibitem[\protect\citeauthoryear{Donninger}{2024}]{Don24}
\begin{barticle}
\bauthor{\bsnm{Donninger}, \binits{R.}}:
\batitle{Spectral theory and self-similar blowup in wave equations}.
\bjtitle{Bull. Amer. Math. Soc. (N.S.)}
\bvolume{61}(\bissue{4}),
\bfpage{659}--\blpage{685}
(\byear{2024})
\doiurl{10.1090/bull/1845}
\end{barticle}
\endbibitem

%%% 8
\bibitem[\protect\citeauthoryear{Donninger et~al.}{2012}]{donn}
\begin{barticle}
\bauthor{\bsnm{Donninger}, \binits{R.}},
\bauthor{\bsnm{Sch\"orkhuber}, \binits{B.}},
\bauthor{\bsnm{Aichelburg}, \binits{P.C.}}:
\batitle{{On stable self-similar blow up for equivariant wave maps: the linearized problem}}.
\bjtitle{Ann. Henri Poincar\'e}
\bvolume{13}(\bissue{1}),
\bfpage{103}--\blpage{144}
(\byear{2012})
\end{barticle}
\endbibitem

%%% 9
\bibitem[\protect\citeauthoryear{Donninger}{2011}]{donn2}
\begin{barticle}
\bauthor{\bsnm{Donninger}, \binits{R.}}:
\batitle{{On stable self-similar blowup for equivariant wave maps}}.
\bjtitle{Comm. Pure Appl. Math.}
\bvolume{64}(\bissue{8}),
\bfpage{1095}--\blpage{1147}
(\byear{2011})
\end{barticle}
\endbibitem

%%% 10
\bibitem[\protect\citeauthoryear{Donninger and Sch\"orkhuber}{2012}]{DonSch12}
\begin{barticle}
\bauthor{\bsnm{Donninger}, \binits{R.}},
\bauthor{\bsnm{Sch\"orkhuber}, \binits{B.}}:
\batitle{Stable self-similar blow up for energy subcritical wave equations}.
\bjtitle{Dyn. Partial Differ. Equ.}
\bvolume{9}(\bissue{1}),
\bfpage{63}--\blpage{87}
(\byear{2012})
\doiurl{10.4310/DPDE.2012.v9.n1.a3}
\end{barticle}
\endbibitem

%%% 11
\bibitem[\protect\citeauthoryear{Donninger and Sch\"orkhuber}{2014}]{DonSch14}
\begin{barticle}
\bauthor{\bsnm{Donninger}, \binits{R.}},
\bauthor{\bsnm{Sch\"orkhuber}, \binits{B.}}:
\batitle{Stable blow up dynamics for energy supercritical wave equations}.
\bjtitle{Trans. Amer. Math. Soc.}
\bvolume{366}(\bissue{4}),
\bfpage{2167}--\blpage{2189}
(\byear{2014})
\doiurl{10.1090/S0002-9947-2013-06038-2}
\end{barticle}
\endbibitem

%%% 12
\bibitem[\protect\citeauthoryear{Donninger and Sch\"orkhuber}{2017}]{DonSch17}
\begin{barticle}
\bauthor{\bsnm{Donninger}, \binits{R.}},
\bauthor{\bsnm{Sch\"orkhuber}, \binits{B.}}:
\batitle{Stable blowup for wave equations in odd space dimensions}.
\bjtitle{Ann. Inst. H. Poincar\'e{} C Anal. Non Lin\'eaire}
\bvolume{34}(\bissue{5}),
\bfpage{1181}--\blpage{1213}
(\byear{2017})
\doiurl{10.1016/j.anihpc.2016.09.005}
\end{barticle}
\endbibitem

%%% 13
\bibitem[\protect\citeauthoryear{Chatzikaleas et~al.}{2017}]{ChaDon17}
\begin{barticle}
\bauthor{\bsnm{Chatzikaleas}, \binits{A.}},
\bauthor{\bsnm{Donninger}, \binits{R.}},
\bauthor{\bsnm{Glogi\'c}, \binits{I.}}:
\batitle{On blowup of co-rotational wave maps in odd space dimensions}.
\bjtitle{J. Differential Equations}
\bvolume{263}(\bissue{8}),
\bfpage{5090}--\blpage{5119}
(\byear{2017})
\doiurl{10.1016/j.jde.2017.06.011}
\end{barticle}
\endbibitem

%%% 14
\bibitem[\protect\citeauthoryear{Donninger and Glogi\'c}{2019}]{DonGlo19}
\begin{barticle}
\bauthor{\bsnm{Donninger}, \binits{R.}},
\bauthor{\bsnm{Glogi\'c}, \binits{I.}}:
\batitle{On the existence and stability of blowup for wave maps into a negatively curved target}.
\bjtitle{Anal. PDE}
\bvolume{12}(\bissue{2}),
\bfpage{389}--\blpage{416}
(\byear{2019})
\doiurl{10.2140/apde.2019.12.389}
\end{barticle}
\endbibitem

%%% 15
\bibitem[\protect\citeauthoryear{Glogi\'c}{2025}]{Glo25}
\begin{barticle}
\bauthor{\bsnm{Glogi\'c}, \binits{I.}}:
\batitle{Globally stable blowup profile for supercritical wave maps in all dimensions}.
\bjtitle{Calc. Var. Partial Differential Equations}
\bvolume{64}(\bissue{2}),
\bfpage{46}
(\byear{2025})
\doiurl{10.1007/s00526-024-02901-7}
\end{barticle}
\endbibitem

%%% 16
\bibitem[\protect\citeauthoryear{Donninger}{2014}]{Don14}
\begin{barticle}
\bauthor{\bsnm{Donninger}, \binits{R.}}:
\batitle{Stable self-similar blowup in energy supercritical {Y}ang-{M}ills theory}.
\bjtitle{Math. Z.}
\bvolume{278}(\bissue{3-4}),
\bfpage{1005}--\blpage{1032}
(\byear{2014})
\doiurl{10.1007/s00209-014-1344-0}
\end{barticle}
\endbibitem

%%% 17
\bibitem[\protect\citeauthoryear{Costin et~al.}{2016}]{CosDonGloHua16}
\begin{barticle}
\bauthor{\bsnm{Costin}, \binits{O.}},
\bauthor{\bsnm{Donninger}, \binits{R.}},
\bauthor{\bsnm{Glogi\'c}, \binits{I.}},
\bauthor{\bsnm{Huang}, \binits{M.}}:
\batitle{On the stability of self-similar solutions to nonlinear wave equations}.
\bjtitle{Comm. Math. Phys.}
\bvolume{343}(\bissue{1}),
\bfpage{299}--\blpage{310}
(\byear{2016})
\doiurl{10.1007/s00220-016-2588-9}
\end{barticle}
\endbibitem

%%% 18
\bibitem[\protect\citeauthoryear{Glogi\'c}{2022}]{Glo22}
\begin{barticle}
\bauthor{\bsnm{Glogi\'c}, \binits{I.}}:
\batitle{Stable blowup for the supercritical hyperbolic {Y}ang-{M}ills equations}.
\bjtitle{Adv. Math.}
\bvolume{408},
\bfpage{108633}--\blpage{52}
(\byear{2022})
\doiurl{10.1016/j.aim.2022.108633}
\end{barticle}
\endbibitem

%%% 19
\bibitem[\protect\citeauthoryear{Glogi\'c}{2024}]{Glo24}
\begin{barticle}
\bauthor{\bsnm{Glogi\'c}, \binits{I.}}:
\batitle{Global-in-space stability of singularity formation for {Y}ang-{M}ills fields in higher dimensions}.
\bjtitle{J. Differential Equations}
\bvolume{408},
\bfpage{140}--\blpage{165}
(\byear{2024})
\doiurl{10.1016/j.jde.2024.06.035}
\end{barticle}
\endbibitem

%%% 20
\bibitem[\protect\citeauthoryear{McNulty}{2020}]{McN20}
\begin{barticle}
\bauthor{\bsnm{McNulty}, \binits{M.}}:
\batitle{Development of singularities of the {S}kyrme model}.
\bjtitle{J. Hyperbolic Differ. Equ.}
\bvolume{17}(\bissue{1}),
\bfpage{61}--\blpage{73}
(\byear{2020})
\doiurl{10.1142/S0219891620500022}
\end{barticle}
\endbibitem

%%% 21
\bibitem[\protect\citeauthoryear{Chen et~al.}{2023}]{CheMcNSch23}
\begin{botherref}
\oauthor{\bsnm{Chen}, \binits{P.-N.}},
\oauthor{\bsnm{McNulty}, \binits{M.}},
\oauthor{\bsnm{Schörkhuber}, \binits{B.}}:
Singularity formation for the higher dimensional Skyrme model in the strong field limit
(2023).
\url{https://arxiv.org/abs/2310.07042}
\end{botherref}
\endbibitem

%%% 22
\bibitem[\protect\citeauthoryear{McNulty}{2024}]{McN24}
\begin{botherref}
\oauthor{\bsnm{McNulty}, \binits{M.}}:
Singularity formation for the higher dimensional Skyrme model
(2024).
\url{https://arxiv.org/abs/2408.15345}
\end{botherref}
\endbibitem

%%% 23
\bibitem[\protect\citeauthoryear{Merle and Zaag}{2007}]{MerZaa07}
\begin{barticle}
\bauthor{\bsnm{Merle}, \binits{F.}},
\bauthor{\bsnm{Zaag}, \binits{H.}}:
\batitle{Existence and universality of the blow-up profile for the semilinear wave equation in one space dimension}.
\bjtitle{J. Funct. Anal.}
\bvolume{253}(\bissue{1}),
\bfpage{43}--\blpage{121}
(\byear{2007})
\doiurl{10.1016/j.jfa.2007.03.007}
\end{barticle}
\endbibitem

%%% 24
\bibitem[\protect\citeauthoryear{Donninger and Sch\"orkhuber}{2016}]{DonSch16}
\begin{barticle}
\bauthor{\bsnm{Donninger}, \binits{R.}},
\bauthor{\bsnm{Sch\"orkhuber}, \binits{B.}}:
\batitle{On blowup in supercritical wave equations}.
\bjtitle{Comm. Math. Phys.}
\bvolume{346}(\bissue{3}),
\bfpage{907}--\blpage{943}
(\byear{2016})
\doiurl{10.1007/s00220-016-2610-2}
\end{barticle}
\endbibitem

%%% 25
\bibitem[\protect\citeauthoryear{Merle and Zaag}{2016}]{MerZaa16}
\begin{barticle}
\bauthor{\bsnm{Merle}, \binits{F.}},
\bauthor{\bsnm{Zaag}, \binits{H.}}:
\batitle{Dynamics near explicit stationary solutions in similarity variables for solutions of a semilinear wave equation in higher dimensions}.
\bjtitle{Trans. Amer. Math. Soc.}
\bvolume{368}(\bissue{1}),
\bfpage{27}--\blpage{87}
(\byear{2016})
\doiurl{10.1090/tran/6450}
\end{barticle}
\endbibitem

%%% 26
\bibitem[\protect\citeauthoryear{Chatzikaleas and Donninger}{2019}]{ChaDon19}
\begin{barticle}
\bauthor{\bsnm{Chatzikaleas}, \binits{A.}},
\bauthor{\bsnm{Donninger}, \binits{R.}}:
\batitle{Stable blowup for the cubic wave equation in higher dimensions}.
\bjtitle{J. Differential Equations}
\bvolume{266}(\bissue{10}),
\bfpage{6809}--\blpage{6865}
(\byear{2019})
\doiurl{10.1016/j.jde.2018.11.016}
\end{barticle}
\endbibitem

%%% 27
\bibitem[\protect\citeauthoryear{Glogi\'c and Sch\"orkhuber}{2021}]{GloSch21}
\begin{barticle}
\bauthor{\bsnm{Glogi\'c}, \binits{I.}},
\bauthor{\bsnm{Sch\"orkhuber}, \binits{B.}}:
\batitle{Co-dimension one stable blowup for the supercritical cubic wave equation}.
\bjtitle{Adv. Math.}
\bvolume{390},
\bfpage{107930}--\blpage{79}
(\byear{2021})
\doiurl{10.1016/j.aim.2021.107930}
\end{barticle}
\endbibitem

%%% 28
\bibitem[\protect\citeauthoryear{Csobo et~al.}{2024}]{CsoGloSch24}
\begin{barticle}
\bauthor{\bsnm{Csobo}, \binits{E.}},
\bauthor{\bsnm{Glogi\'c}, \binits{I.}},
\bauthor{\bsnm{Sch\"orkhuber}, \binits{B.}}:
\batitle{On blowup for the supercritical quadratic wave equation}.
\bjtitle{Anal. PDE}
\bvolume{17}(\bissue{2}),
\bfpage{617}--\blpage{680}
(\byear{2024})
\doiurl{10.2140/apde.2024.17.617}
\end{barticle}
\endbibitem

%%% 29
\bibitem[\protect\citeauthoryear{Ostermann}{2024}]{Ost24}
\begin{barticle}
\bauthor{\bsnm{Ostermann}, \binits{M.}}:
\batitle{Stable blowup for focusing semilinear wave equations in all dimensions}.
\bjtitle{Trans. Amer. Math. Soc.}
\bvolume{377}(\bissue{7}),
\bfpage{4727}--\blpage{4778}
(\byear{2024})
\doiurl{10.1090/tran/9069}
\end{barticle}
\endbibitem

%%% 30
\bibitem[\protect\citeauthoryear{Bizo\'n and Tabor}{2001}]{BizTab01}
\begin{barticle}
\bauthor{\bsnm{Bizo\'n}, \binits{P.}},
\bauthor{\bsnm{Tabor}, \binits{Z.}}:
\batitle{On blowup of {Y}ang-{M}ills fields}.
\bjtitle{Phys. Rev. D (3)}
\bvolume{64}(\bissue{12}),
\bfpage{121701}--\blpage{4}
(\byear{2001})
\doiurl{10.1103/PhysRevD.64.121701}
\end{barticle}
\endbibitem

%%% 31
\bibitem[\protect\citeauthoryear{Bizo\'n}{2002}]{Biz02}
\begin{barticle}
\bauthor{\bsnm{Bizo\'n}, \binits{P.}}:
\batitle{Formation of singularities in {Y}ang-{M}ills equations}.
\bjtitle{Acta Phys. Polon. B}
\bvolume{33}(\bissue{7}),
\bfpage{1893}--\blpage{1922}
(\byear{2002})
\end{barticle}
\endbibitem

%%% 32
\bibitem[\protect\citeauthoryear{Bizo\'n et~al.}{2004}]{BizChmTab04}
\begin{barticle}
\bauthor{\bsnm{Bizo\'n}, \binits{P.}},
\bauthor{\bsnm{Chmaj}, \binits{T.}},
\bauthor{\bsnm{Tabor}, \binits{Z.a.}}:
\batitle{On blowup for semilinear wave equations with a focusing nonlinearity}.
\bjtitle{Nonlinearity}
\bvolume{17}(\bissue{6}),
\bfpage{2187}--\blpage{2201}
(\byear{2004})
\doiurl{10.1088/0951-7715/17/6/009}
\end{barticle}
\endbibitem

%%% 33
\bibitem[\protect\citeauthoryear{Bizo\'n}{2005}]{Biz05}
\begin{barticle}
\bauthor{\bsnm{Bizo\'n}, \binits{P.}}:
\batitle{An unusual eigenvalue problem}.
\bjtitle{Acta Phys. Polon. B}
\bvolume{36}(\bissue{1}),
\bfpage{5}--\blpage{15}
(\byear{2005})
\end{barticle}
\endbibitem

%%% 34
\bibitem[\protect\citeauthoryear{Donninger}{2017}]{Don17}
\begin{barticle}
\bauthor{\bsnm{Donninger}, \binits{R.}}:
\batitle{Strichartz estimates in similarity coordinates and stable blowup for the critical wave equation}.
\bjtitle{Duke Math. J.}
\bvolume{166}(\bissue{9}),
\bfpage{1627}--\blpage{1683}
(\byear{2017})
\doiurl{10.1215/00127094-0000009X}
\end{barticle}
\endbibitem

%%% 35
\bibitem[\protect\citeauthoryear{Donninger and Rao}{2020}]{DonRao19}
\begin{barticle}
\bauthor{\bsnm{Donninger}, \binits{R.}},
\bauthor{\bsnm{Rao}, \binits{Z.}}:
\batitle{Blowup stability at optimal regularity for the critical wave equation}.
\bjtitle{Adv. Math.}
\bvolume{370},
\bfpage{107219}--\blpage{81}
(\byear{2020})
\doiurl{10.1016/j.aim.2020.107219}
\end{barticle}
\endbibitem

%%% 36
\bibitem[\protect\citeauthoryear{Wallauch}{2023}]{Wal23}
\begin{barticle}
\bauthor{\bsnm{Wallauch}, \binits{D.}}:
\batitle{Strichartz estimates and blowup stability for energy critical nonlinear wave equations}.
\bjtitle{Trans. Amer. Math. Soc.}
\bvolume{376}(\bissue{6}),
\bfpage{4321}--\blpage{4360}
(\byear{2023})
\doiurl{10.1090/tran/8879}
\end{barticle}
\endbibitem

%%% 37
\bibitem[\protect\citeauthoryear{Donninger and Wallauch}{2023a}]{DonWal23}
\begin{barticle}
\bauthor{\bsnm{Donninger}, \binits{R.}},
\bauthor{\bsnm{Wallauch}, \binits{D.}}:
\batitle{Optimal blowup stability for supercritical wave maps}.
\bjtitle{Adv. Math.}
\bvolume{433},
\bfpage{109291}--\blpage{86}
(\byear{2023})
\doiurl{10.1016/j.aim.2023.109291}
\end{barticle}
\endbibitem

%%% 38
\bibitem[\protect\citeauthoryear{Donninger and Wallauch}{2023b}]{DonWal25}
\begin{botherref}
\oauthor{\bsnm{Donninger}, \binits{R.}},
\oauthor{\bsnm{Wallauch}, \binits{D.}}:
Optimal blowup stability for three-dimensional wave maps
(2023).
\url{https://arxiv.org/abs/2212.08374}
\end{botherref}
\endbibitem

%%% 39
\bibitem[\protect\citeauthoryear{Biernat et~al.}{2021}]{BieDonSch21}
\begin{botherref}
\oauthor{\bsnm{Biernat}, \binits{P.}},
\oauthor{\bsnm{Donninger}, \binits{R.}},
\oauthor{\bsnm{Sch\"orkhuber}, \binits{B.}}:
Hyperboloidal similarity coordinates and a globally stable blowup profile for supercritical wave maps.
Int. Math. Res. Not. IMRN
(21),
16530--16591
(2021)
\doiurl{10.1093/imrn/rnz286}
\end{botherref}
\endbibitem

%%% 40
\bibitem[\protect\citeauthoryear{Donninger and Ostermann}{2023}]{DonOst23}
\begin{barticle}
\bauthor{\bsnm{Donninger}, \binits{R.}},
\bauthor{\bsnm{Ostermann}, \binits{M.}}:
\batitle{A globally stable self-similar blowup profile in energy supercritical {Y}ang-{M}ills theory}.
\bjtitle{Comm. Partial Differential Equations}
\bvolume{48}(\bissue{9}),
\bfpage{1148}--\blpage{1213}
(\byear{2023})
\doiurl{10.1080/03605302.2023.2263208}
\end{barticle}
\endbibitem

%%% 41
\bibitem[\protect\citeauthoryear{Donninger and Ostermann}{2024}]{DonOst24}
\begin{botherref}
\oauthor{\bsnm{Donninger}, \binits{R.}},
\oauthor{\bsnm{Ostermann}, \binits{M.}}:
On stable self-similar blowup for corotational wave maps and equivariant Yang-Mills connections
(2024).
\url{https://arxiv.org/abs/2409.14733}
\end{botherref}
\endbibitem

%%% 42
\bibitem[\protect\citeauthoryear{Chen et~al.}{2024}]{CheDonGloMcNSch24}
\begin{barticle}
\bauthor{\bsnm{Chen}, \binits{P.-N.}},
\bauthor{\bsnm{Donninger}, \binits{R.}},
\bauthor{\bsnm{Glogi\'c}, \binits{I.}},
\bauthor{\bsnm{McNulty}, \binits{M.}},
\bauthor{\bsnm{Sch\"orkhuber}, \binits{B.}}:
\batitle{Co-dimension one stable blowup for the quadratic wave equation beyond the light cone}.
\bjtitle{Comm. Math. Phys.}
\bvolume{405}(\bissue{2}),
\bfpage{34}--\blpage{46}
(\byear{2024})
\doiurl{10.1007/s00220-023-04888-2}
\end{barticle}
\endbibitem

%%% 43
\bibitem[\protect\citeauthoryear{Bringmann}{2020}]{Bri20}
\begin{barticle}
\bauthor{\bsnm{Bringmann}, \binits{B.}}:
\batitle{Stable blowup for the focusing energy critical nonlinear wave equation under random perturbations}.
\bjtitle{Comm. Partial Differential Equations}
\bvolume{45}(\bissue{12}),
\bfpage{1755}--\blpage{1777}
(\byear{2020})
\doiurl{10.1080/03605302.2020.1803356}
\end{barticle}
\endbibitem

%%% 44
\bibitem[\protect\citeauthoryear{Krieger et~al.}{2008}]{KriSchTat08}
\begin{barticle}
\bauthor{\bsnm{Krieger}, \binits{J.}},
\bauthor{\bsnm{Schlag}, \binits{W.}},
\bauthor{\bsnm{Tataru}, \binits{D.}}:
\batitle{Renormalization and blow up for charge one equivariant critical wave maps}.
\bjtitle{Invent. Math.}
\bvolume{171}(\bissue{3}),
\bfpage{543}--\blpage{615}
(\byear{2008})
\doiurl{10.1007/s00222-007-0089-3}
\end{barticle}
\endbibitem

%%% 45
\bibitem[\protect\citeauthoryear{Krieger et~al.}{2009a}]{KriSchTat09}
\begin{barticle}
\bauthor{\bsnm{Krieger}, \binits{J.}},
\bauthor{\bsnm{Schlag}, \binits{W.}},
\bauthor{\bsnm{Tataru}, \binits{D.}}:
\batitle{Slow blow-up solutions for the {$H^1(\Bbb R^3)$} critical focusing semilinear wave equation}.
\bjtitle{Duke Math. J.}
\bvolume{147}(\bissue{1}),
\bfpage{1}--\blpage{53}
(\byear{2009})
\doiurl{10.1215/00127094-2009-005}
\end{barticle}
\endbibitem

%%% 46
\bibitem[\protect\citeauthoryear{Krieger et~al.}{2009b}]{KriSchTat09a}
\begin{barticle}
\bauthor{\bsnm{Krieger}, \binits{J.}},
\bauthor{\bsnm{Schlag}, \binits{W.}},
\bauthor{\bsnm{Tataru}, \binits{D.}}:
\batitle{Renormalization and blow up for the critical {Y}ang-{M}ills problem}.
\bjtitle{Adv. Math.}
\bvolume{221}(\bissue{5}),
\bfpage{1445}--\blpage{1521}
(\byear{2009})
\doiurl{10.1016/j.aim.2009.02.017}
\end{barticle}
\endbibitem

%%% 47
\bibitem[\protect\citeauthoryear{Rodnianski and Sterbenz}{2010}]{RodSte10}
\begin{barticle}
\bauthor{\bsnm{Rodnianski}, \binits{I.}},
\bauthor{\bsnm{Sterbenz}, \binits{J.}}:
\batitle{On the formation of singularities in the critical {${\rm O}(3)$} {$\sigma$}-model}.
\bjtitle{Ann. of Math. (2)}
\bvolume{172}(\bissue{1}),
\bfpage{187}--\blpage{242}
(\byear{2010})
\doiurl{10.4007/annals.2010.172.187}
\end{barticle}
\endbibitem

%%% 48
\bibitem[\protect\citeauthoryear{Rapha\"el and Rodnianski}{2012}]{RapRod12}
\begin{barticle}
\bauthor{\bsnm{Rapha\"el}, \binits{P.}},
\bauthor{\bsnm{Rodnianski}, \binits{I.}}:
\batitle{Stable blow up dynamics for the critical co-rotational wave maps and equivariant {Y}ang-{M}ills problems}.
\bjtitle{Publ. Math. Inst. Hautes \'Etudes Sci.}
\bvolume{115},
\bfpage{1}--\blpage{122}
(\byear{2012})
\doiurl{10.1007/s10240-011-0037-z}
\end{barticle}
\endbibitem

%%% 49
\bibitem[\protect\citeauthoryear{Hillairet and Rapha\"el}{2012}]{HilRap12}
\begin{barticle}
\bauthor{\bsnm{Hillairet}, \binits{M.}},
\bauthor{\bsnm{Rapha\"el}, \binits{P.}}:
\batitle{Smooth type {II} blow-up solutions to the four-dimensional energy-critical wave equation}.
\bjtitle{Anal. PDE}
\bvolume{5}(\bissue{4}),
\bfpage{777}--\blpage{829}
(\byear{2012})
\doiurl{10.2140/apde.2012.5.777}
\end{barticle}
\endbibitem

%%% 50
\bibitem[\protect\citeauthoryear{Jendrej}{2017}]{Jen17}
\begin{barticle}
\bauthor{\bsnm{Jendrej}, \binits{J.}}:
\batitle{Construction of type {II} blow-up solutions for the energy-critical wave equation in dimension 5}.
\bjtitle{J. Funct. Anal.}
\bvolume{272}(\bissue{3}),
\bfpage{866}--\blpage{917}
(\byear{2017})
\doiurl{10.1016/j.jfa.2016.10.019}
\end{barticle}
\endbibitem

%%% 51
\bibitem[\protect\citeauthoryear{Ghoul et~al.}{2018}]{GhoIbrNgu18}
\begin{barticle}
\bauthor{\bsnm{Ghoul}, \binits{T.}},
\bauthor{\bsnm{Ibrahim}, \binits{S.}},
\bauthor{\bsnm{Nguyen}, \binits{V.T.}}:
\batitle{Construction of type {II} blowup solutions for the 1-corotational energy supercritical wave maps}.
\bjtitle{J. Differential Equations}
\bvolume{265}(\bissue{7}),
\bfpage{2968}--\blpage{3047}
(\byear{2018})
\doiurl{10.1016/j.jde.2018.04.058}
\end{barticle}
\endbibitem

%%% 52
\bibitem[\protect\citeauthoryear{Collot}{2018}]{Col18}
\begin{barticle}
\bauthor{\bsnm{Collot}, \binits{C.}}:
\batitle{Type {II} blow up manifolds for the energy supercritical semilinear wave equation}.
\bjtitle{Mem. Amer. Math. Soc.}
\bvolume{252}(\bissue{1205}),
\bfpage{163}
(\byear{2018})
\doiurl{10.1090/memo/1205}
\end{barticle}
\endbibitem

%%% 53
\bibitem[\protect\citeauthoryear{Krieger and Miao}{2020}]{KriMia20}
\begin{barticle}
\bauthor{\bsnm{Krieger}, \binits{J.}},
\bauthor{\bsnm{Miao}, \binits{S.}}:
\batitle{On the stability of blowup solutions for the critical corotational wave-map problem}.
\bjtitle{Duke Math. J.}
\bvolume{169}(\bissue{3}),
\bfpage{435}--\blpage{532}
(\byear{2020})
\doiurl{10.1215/00127094-2019-0053}
\end{barticle}
\endbibitem

%%% 54
\bibitem[\protect\citeauthoryear{Burzio and Krieger}{2022}]{BurKri22}
\begin{barticle}
\bauthor{\bsnm{Burzio}, \binits{S.}},
\bauthor{\bsnm{Krieger}, \binits{J.}}:
\batitle{Type {II} blow up solutions with optimal stability properties for the critical focussing nonlinear wave equation on {$\Bbb R^{3+1}$}}.
\bjtitle{Mem. Amer. Math. Soc.}
\bvolume{278}(\bissue{1369}),
\bfpage{75}
(\byear{2022})
\doiurl{10.1090/memo/1369}
\end{barticle}
\endbibitem

%%% 55
\bibitem[\protect\citeauthoryear{Jendrej and Krieger}{2025}]{JenKri25}
\begin{botherref}
\oauthor{\bsnm{Jendrej}, \binits{J.}},
\oauthor{\bsnm{Krieger}, \binits{J.}}:
Concentric bubbles concentrating in finite time for the energy critical wave maps equation
(2025).
\url{https://arxiv.org/abs/2501.08396}
\end{botherref}
\endbibitem

%%% 56
\bibitem[\protect\citeauthoryear{Krieger et~al.}{2024}]{KriMiaSch24}
\begin{botherref}
\oauthor{\bsnm{Krieger}, \binits{J.}},
\oauthor{\bsnm{Miao}, \binits{S.}},
\oauthor{\bsnm{Schlag}, \binits{W.}}:
A stability theory beyond the co-rotational setting for critical Wave Maps blow up
(2024).
\url{https://arxiv.org/abs/2009.08843}
\end{botherref}
\endbibitem

%%% 57
\bibitem[\protect\citeauthoryear{Klainerman and Selberg}{1997}]{klainermanselberg1}
\begin{barticle}
\bauthor{\bsnm{Klainerman}, \binits{S.}},
\bauthor{\bsnm{Selberg}, \binits{S.}}:
\batitle{{Remark on the optimal regularity for equations of wave maps type}}.
\bjtitle{Communications in Partial Differential Equations}
\bvolume{22}(\bissue{5-6}),
\bfpage{99}--\blpage{133}
(\byear{1997})
\end{barticle}
\endbibitem

%%% 58
\bibitem[\protect\citeauthoryear{Klainerman and Machedon}{1996}]{klainermanmachedon}
\begin{barticle}
\bauthor{\bsnm{Klainerman}, \binits{S.}},
\bauthor{\bsnm{Machedon}, \binits{M.}}:
\batitle{{Smoothing estimates for null forms and applications}}.
\bjtitle{Duke Math. J}
\bvolume{81}(\bissue{1}),
\bfpage{99}--\blpage{133}
(\byear{1996})
\end{barticle}
\endbibitem

%%% 59
\bibitem[\protect\citeauthoryear{Klainerman and Machedon}{}]{klainerman1}
\begin{botherref}
\oauthor{\bsnm{Klainerman}, \binits{S.}},
\oauthor{\bsnm{Machedon}, \binits{M.}}:
{Space-time estimates for null forms and the local existence theorem}.
Communications on Pure and Applied Mathematics
\textbf{46}(9),
1221--1268
\end{botherref}
\endbibitem

%%% 60
\bibitem[\protect\citeauthoryear{Klainerman and Machedon}{1997}]{klainerman2}
\begin{barticle}
\bauthor{\bsnm{Klainerman}, \binits{S.}},
\bauthor{\bsnm{Machedon}, \binits{M.}}:
\batitle{{On the regularity properties of a model problem related to wave maps}}.
\bjtitle{Duke Math. J.}
\bvolume{87}(\bissue{3}),
\bfpage{553}--\blpage{589}
(\byear{1997})
\end{barticle}
\endbibitem

%%% 61
\bibitem[\protect\citeauthoryear{Tataru}{1998}]{tataru1}
\begin{barticle}
\bauthor{\bsnm{Tataru}, \binits{D.}}:
\batitle{{Local and global results for wave maps I}}.
\bjtitle{Communications in Partial Differential Equations}
\bvolume{23}(\bissue{9-10}),
\bfpage{1781}--\blpage{1793}
(\byear{1998})
\end{barticle}
\endbibitem

%%% 62
\bibitem[\protect\citeauthoryear{Tataru}{2001}]{tataru2}
\begin{barticle}
\bauthor{\bsnm{Tataru}, \binits{D.}}:
\batitle{{On Global Existence and Scattering for the Wave Maps Equation}}.
\bjtitle{American Journal of Mathematics}
\bvolume{123}(\bissue{1}),
\bfpage{37}--\blpage{77}
(\byear{2001})
\end{barticle}
\endbibitem

%%% 63
\bibitem[\protect\citeauthoryear{Tao}{2001a}]{tao1}
\begin{barticle}
\bauthor{\bsnm{Tao}, \binits{T.}}:
\batitle{{Multilinear weighted convolution of $L^2$ functions, and applications to nonlinear dispersive equations}}.
\bjtitle{American Journal of Mathematics}
\bvolume{123}(\bissue{5}),
\bfpage{839}--\blpage{908}
(\byear{2001})
\end{barticle}
\endbibitem

%%% 64
\bibitem[\protect\citeauthoryear{Tao}{2001b}]{tao2}
\begin{barticle}
\bauthor{\bsnm{Tao}, \binits{T.}}:
\batitle{{Global Regularity of Wave Maps II: Small Energy in Two Dimensions}}.
\bjtitle{Communications in Mathematical Physics}
\bvolume{224}(\bissue{2}),
\bfpage{443}--\blpage{544}
(\byear{2001})
\end{barticle}
\endbibitem

%%% 65
\bibitem[\protect\citeauthoryear{Klainerman and Rodnianski}{2001}]{klainrod}
\begin{barticle}
\bauthor{\bsnm{Klainerman}, \binits{S.}},
\bauthor{\bsnm{Rodnianski}, \binits{I.}}:
\batitle{{On the global regularity of wave maps in the critical Sobolev norm}}.
\bjtitle{International Mathematics Research Notices}
\bvolume{2001}(\bissue{13}),
\bfpage{655}--\blpage{677}
(\byear{2001})
\end{barticle}
\endbibitem

%%% 66
\bibitem[\protect\citeauthoryear{Nahmod et~al.}{2003}]{nahmod2001well}
\begin{botherref}
\oauthor{\bsnm{Nahmod}, \binits{A.}},
\oauthor{\bsnm{Stefanov}, \binits{A.}},
\oauthor{\bsnm{Uhlenbeck}, \binits{K.}}:
{On the Well-Posedness of the Wave Map Problem in High Dimensions}.
Communications in Analysis and Geometry
\textbf{11}(1)
(2003)
\end{botherref}
\endbibitem

%%% 67
\bibitem[\protect\citeauthoryear{Shatah and Struwe}{2002}]{struwe}
\begin{barticle}
\bauthor{\bsnm{Shatah}, \binits{J.}},
\bauthor{\bsnm{Struwe}, \binits{M.}}:
\batitle{{The Cauchy problem for wave maps}}.
\bjtitle{International Mathematics Research Notices}
\bvolume{2002}(\bissue{11}),
\bfpage{555}--\blpage{571}
(\byear{2002})
\end{barticle}
\endbibitem

%%% 68
\bibitem[\protect\citeauthoryear{Krieger}{2003}]{krieger1}
\begin{barticle}
\bauthor{\bsnm{Krieger}, \binits{J.}}:
\batitle{{Global regularity of wave maps from $\mathbb{R}^{3+ 1}$ to surfaces}}.
\bjtitle{Communications in mathematical physics}
\bvolume{238}(\bissue{1-2}),
\bfpage{333}--\blpage{366}
(\byear{2003})
\end{barticle}
\endbibitem

%%% 69
\bibitem[\protect\citeauthoryear{Krieger}{2004}]{krieger2}
\begin{barticle}
\bauthor{\bsnm{Krieger}, \binits{J.}}:
\batitle{{Global regularity of wave maps from $\mathbb{R}^{2+1}$ to $\mathbb{H}^2$. Small energy}}.
\bjtitle{Communications in mathematical physics}
\bvolume{250}(\bissue{3}),
\bfpage{507}--\blpage{580}
(\byear{2004})
\end{barticle}
\endbibitem

%%% 70
\bibitem[\protect\citeauthoryear{Tataru}{2005}]{tataru3}
\begin{barticle}
\bauthor{\bsnm{Tataru}, \binits{D.}}:
\batitle{{Rough solutions for the wave maps equation}}.
\bjtitle{American Journal of Mathematics}
\bvolume{127}(\bissue{2}),
\bfpage{293}--\blpage{377}
(\byear{2005})
\end{barticle}
\endbibitem

%%% 71
\bibitem[\protect\citeauthoryear{Krieger}{2007}]{krieger}
\begin{barticle}
\bauthor{\bsnm{Krieger}, \binits{J.}}:
\batitle{{Global Regularity and Singularity Development for Wave Maps}}.
\bjtitle{Surveys in differential geometry}
\bvolume{XII},
\bfpage{167}--\blpage{201}
(\byear{2007})
\end{barticle}
\endbibitem

%%% 72
\bibitem[\protect\citeauthoryear{Hall}{2015}]{halllie}
\begin{bbook}
\bauthor{\bsnm{Hall}, \binits{B.}}:
\bbtitle{Lie Groups, {L}ie Algebras, and Representations},
\bedition{2}nd edn.
\bsertitle{Graduate Texts in Mathematics},
vol. \bseriesno{222},
p. \bfpage{449}.
\bpublisher{Springer},
\blocation{Cham}
(\byear{2015}).
\bcomment{An elementary introduction}
\end{bbook}
\endbibitem

%%% 73
\bibitem[\protect\citeauthoryear{Hall}{2013}]{hall2}
\begin{bbook}
\bauthor{\bsnm{Hall}, \binits{B.C.}}:
\bbtitle{Quantum Theory for Mathematicians}.
\bsertitle{Graduate Texts in Mathematics},
vol. \bseriesno{267},
p. \bfpage{554}.
\bpublisher{Springer},
\blocation{New York}
(\byear{2013})
\end{bbook}
\endbibitem

%%% 74
\bibitem[\protect\citeauthoryear{Tanabashi}{2018}]{clebschgordan}
\begin{barticle}
\bauthor{\bsnm{Tanabashi}, \binits{M.e.a.}}:
\batitle{{Review of Particle Physics}}.
\bjtitle{Phys. Rev. D}
\bvolume{98},
\bfpage{030001}
(\byear{2018})
\end{barticle}
\endbibitem

%%% 75
\bibitem[\protect\citeauthoryear{Teschl}{2012}]{teschl}
\begin{bbook}
\bauthor{\bsnm{Teschl}, \binits{G.}}:
\bbtitle{Ordinary Differential Equations and Dynamical Systems}.
\bsertitle{Graduate Studies in Mathematics},
vol. \bseriesno{140},
p. \bfpage{356}.
\bpublisher{American Mathematical Society},
\blocation{Providence, RI}
(\byear{2012})
\end{bbook}
\endbibitem

%%% 76
\bibitem[\protect\citeauthoryear{Olver et~al.}{2010}]{handbook}
\begin{bbook}
\bauthor{\bsnm{Olver}, \binits{F.W.}},
\bauthor{\bsnm{Lozier}, \binits{D.W.}},
\bauthor{\bsnm{Boisvert}, \binits{R.F.}},
\bauthor{\bsnm{Clark}, \binits{C.W.}}:
\bbtitle{NIST Handbook of Mathematical Functions Hardback and CD-ROM}.
\bpublisher{Cambridge university press},
\blocation{Cambridge}
(\byear{2010})
\end{bbook}
\endbibitem

%%% 77
\bibitem[\protect\citeauthoryear{Glogi{\'c} and Sch{\"o}rkhuber}{2021}]{schoe}
\begin{barticle}
\bauthor{\bsnm{Glogi{\'c}}, \binits{I.}},
\bauthor{\bsnm{Sch{\"o}rkhuber}, \binits{B.}}:
\batitle{{Co-dimension one stable blowup for the supercritical cubic wave equation}}.
\bjtitle{Advances in Mathematics}
\bvolume{390},
\bfpage{107930}
(\byear{2021})
\end{barticle}
\endbibitem

%%% 78
\bibitem[\protect\citeauthoryear{Elaydi}{2005}]{recurrencerelations}
\begin{bbook}
\bauthor{\bsnm{Elaydi}, \binits{S.}}:
\bbtitle{An Introduction to Difference Equations}.
\bsertitle{Undergraduate Texts in Mathematics}.
\bpublisher{Springer},
\blocation{New York}
(\byear{2005})
\end{bbook}
\endbibitem

%%% 79
\bibitem[\protect\citeauthoryear{Wall}{1945}]{wall}
\begin{barticle}
\bauthor{\bsnm{Wall}, \binits{H.}}:
\batitle{{Polynomials whose zeros have negative real parts}}.
\bjtitle{The American Mathematical Monthly}
\bvolume{52}(\bissue{6}),
\bfpage{308}--\blpage{322}
(\byear{1945})
\end{barticle}
\endbibitem

\end{thebibliography}

\end{document}